\documentclass[11pt,a4paper]{amsart}

\usepackage{geometry}
\usepackage{amssymb}
\usepackage{amsthm}
\usepackage{amsmath}
\usepackage{mathdots}
\usepackage{graphicx}
\usepackage{xfrac}
\usepackage{bm}
\usepackage{parskip}
\usepackage{mathtools}
\usepackage{leftindex}

\usepackage[UKenglish]{babel}
\usepackage[dvipsnames]{xcolor}
\usepackage{booktabs}
\usepackage{microtype}
\usepackage{times}
\usepackage[T1]{fontenc}
\usepackage{url}
\usepackage{tikz}

\definecolor{dark-gray}{gray}{0.3}

\usepackage[colorlinks=true]{hyperref}
\hypersetup{
  urlcolor=Blue,
  citecolor=Black,
  linkcolor=dark-gray
}

\tikzset{
  every neuron/.style={
    circle,
    draw,
    minimum size=0.5cm
  }
}

\numberwithin{equation}{section}

\newtheorem{theorem}{Theorem}[section]
\newtheorem{lemma}[theorem]{Lemma}
\newtheorem{proposition}[theorem]{Proposition}
\newtheorem{conjecture}[theorem]{Conjecture}
\newtheorem{corollary}[theorem]{Corollary}

\theoremstyle{definition}

\newtheorem{definition}[theorem]{Definition}
\newtheorem{example}[theorem]{Example}
\newtheorem{remark}[theorem]{Remark}

\renewcommand{\phi}{\varphi}

\newcommand{\R}{\mathbb{R}}
\newcommand{\N}{\mathbb{N}}
\newcommand{\Z}{\mathbb{Z}}
\newcommand{\Q}{\mathbb{Q}}

\newcommand{\cE}{\mathcal{E}}

\newcommand{\cU}{\mathcal{U}}
\newcommand{\boldq}{\mathbf{q}}
\newcommand{\mC}{\mathcal{C}}

\DeclareMathOperator{\md}{\mathrm{mdeg}}
\DeclareMathOperator{\reg}{\mathrm{reg}}
\DeclareMathOperator*{\prob}{\mathbb{P}}
\DeclareMathOperator*{\argmax}{arg\,max}
\DeclareMathOperator{\dist}{dist}
\DeclareMathOperator{\vol}{vol}

\newcommand{\diag}{\operatorname{diag}}
\newcommand{\inter}{\operatorname{int}}
\newcommand{\diff}[1]{\mathrm{d}{#1}}
\newcommand{\econst}{\mathrm{e}}
\newcommand{\Pfaff}{\mathrm{Pfaff}}
\newcommand{\trans}{\intercal}

\title[Pfaffian Tubes and Neural Networks]{Tubular Neighbourhoods of Pfaffian Sets and Applications to Neural Networks}

\author{Paul Lezeau}
\address{(Lezeau) London School of Geometry and Number Theory, Imperial College London, UK}
\email{p.lezeau23@imperial.ac.uk}
\author{Martin Lotz}
\address{(Lotz) Warwick Mathematical Institute, University of Warwick, UK}
\email{martin.lotz@warwick.ac.uk}

\date{}
\begin{document}

\begin{abstract}
We derive local bounds for the volume of tubular neighbourhoods of regular Pfaffian hypersurfaces in terms of the format of their defining functions. Assuming regular pairwise score equalities, we obtain tail bounds for a condition number measuring the label stability of neural network classifiers under uniform and Gaussian sampling. For a single hidden sigmoid layer with rational first-layer weights and affine outputs, we derive bounds polynomial in the width $w$. For a fixed number of classes, input dimension $n$ and lattice constant, the condition-number tail is $O(w^n/t)$ at every positive sampling scale. We also construct compact regular zero sets showing that the width exponent $n$ in the associated Gauss-map degree bound is optimal.
\end{abstract}

\maketitle


\section{Introduction}
Pfaffian functions satisfy triangular systems of first-order partial differential equations with polynomial coefficients. They extend polynomial functions to include transcendental examples such as $\exp$, $\log$ on its domain, and $\tanh$, while retaining effective quantitative bounds on the geometry of their zero sets.

This paper has three main aims. First, we establish local tube-volume bounds for regular Pfaffian hypersurfaces, including unbounded ones. Second, we apply these bounds to the robustness of neural network classifiers, using a condition number defined by distance to the decision boundary. Third, for a single sigmoid layer with rational first-layer weights, we replace the exponential dependence on width in the general Pfaffian estimate by a polynomial bound.

Let $f\colon \R^n\to \R$ be a Pfaffian function with chain length $s$, chain degree $\alpha$ and degree $\beta$ (as defined in Section~\ref{sec:pfaffian}), where $\alpha,\beta\ge1$, and suppose $0$ is a regular value of $f$. For $V=\mathcal Z(f)$ and $X$ uniform in any ball $B(p,\rho)$, Theorem~\ref{thm:prob_bound} gives
\[
\prob\{\dist(X,V)\le\varepsilon\}
\le C_{\alpha,\beta,s,n}
\left[\left(1+(\alpha+\beta)\frac{\varepsilon}{\rho}\right)^n
-\left(1+\frac{\varepsilon}{\rho}\right)^n\right],
\]
where the explicit constant~\eqref{eq:pfaffian-tube-constant} depends only on the format and dimension. For fixed parameters the right-hand side is $O(\varepsilon/\rho)$ as $\varepsilon/\rho\to0$. The proof combines the integral-geometric tube estimate of~\cite{lotz2015volume} with Khovanskii's zero bound~\cite{khovanskii1991fewnomials}. 

Many smooth neural network activation functions, including the logistic sigmoid $\sigma(t)=(1+\econst^{-t})^{-1}$ and $\tanh$, are Pfaffian. A neural network classifier $F\colon\R^n\to\R^m$ assigns an input $x\in \R^n$ to a label, or class, $k$, if $k\in \argmax_j F_j(x)$ (with ties broken arbitrarily), where each $F_j$ is referred to as a score function. Networks built from Pfaffian activations and Pfaffian output maps are Pfaffian, and the decision boundaries for classification problems are semi-Pfaffian. Let $\Sigma$ be the locus of ties for the maximal score. We define the local condition number
\[
\mC_p(x)=\frac{\|x-p\|}{\dist(x,\Sigma)}.
\]
When all pairwise score differences have $0$ as a regular value, the denominator is the infimum of perturbation sizes that change a uniquely predicted label. This measures label stability, independently of whether the label agrees with ground truth. The tube estimate gives explicit $O(1/t)$ bounds for $\prob\{\mC_p(X)>t\}$ both for uniform sampling in $B(p,\rho)$ and for Gaussian sampling $X\sim\mathcal N(p,\varsigma^2 I_n)$ at every positive scale (Theorems~\ref{thm:condition-uniform} and~\ref{thm:condition-gaussian}).

For a sigmoid network with $h$ hidden units, the general bound contains the Khovanskii factor $2^{h(h-1)/2}$. For a single hidden layer,
we improve this substantially. Suppose its width is $w$, its first-layer weight vectors are $a_k\in\Q^n$, and its outputs are affine, optionally followed by softmax. Let $q\in\N_{>0}$ clear the denominators and put $L=\max_k\|qa_k\|_\infty\ge1$. If every pairwise logit difference has $0$ as a regular value, Corollary~\ref{cor:nn-condition-single-layer} gives
\[
\prob\{\mC_p(X)>t\}
\le\min\left\{1,\binom m2 H(n,L)\frac{w^n}{t}\right\},
\qquad t>0,
\]
for an explicit $H(n,L)$, for both sampling models and all centres and positive scales. The polynomial width dependence is at fixed $n,L$. 

Single-hidden-layer networks are a natural setting in which to study width. The universal approximation theorems of Cybenko~\cite{cybenko1989approximation} and Hornik, Stinchcombe, and White~\cite{hornik1989multilayer} show that finite sums of sigmoids are dense in continuous functions on compact sets. Barron's theorem~\cite{barron1993universal} gives an $L^2$-norm approximation error of order $C_f/\sqrt w$ for target functions with finite first Fourier moment $C_f$; see also~\cite{pinkus1999approximation}.

The geometric estimates use representations of a scalar network
\[
f(x)=c_0+\sum_{k=1}^w d_k\sigma(a_k^{\trans}x+b_k).
\]
The tubular neighbourhood bounds depend on the maximal degree $\md(\mathcal Z(f))$ of the generalised Gauss map and on the section degrees $\md_i(\mathcal Z(f))$, defined in Section~\ref{sec:tubular}.
After the substitution $y_j=\econst^{-x_j/q}$, clearing a single product of positive denominators turns its Gauss-fibre equations determining the maximal degree into Laurent polynomial equations with a common zonotope containing their supports. Bernstein's theorem gives
\[
\md(\mathcal Z(f))\le n!L^nw^n.
\]
A compact grid construction shows that the exponent $n$ in this Gauss-map degree bound is optimal (Proposition~\ref{prop:single-layer-lowerbound}). To handle arbitrary affine sections, we keep the $n$ coordinate exponentials as a global Pfaffian chain and pull them back to the section. Its length is independent of $w$, so the standard Khovanskii theorem gives polynomial bounds for all section degrees. Finally, a separate zero count along lines bounds the leading tube coefficient by $O_{n,L}(w^n)$. Extending the degree estimate to greater depth is the subject of Conjecture~\ref{conj:multi-layer-degree}.

A limit of the regular levels $f=\pm\delta$ also gives a full-tube probability bound of $O_{n,L}(w^n\varepsilon/\rho)$ for possibly singular shallow zero sets (Corollary~\ref{cor:single-layer-singular-linear}).

\subsection{Previous and related work}
The problem of bounding the volume of a tubular neighbourhood of a set in terms of its geometric complexity has a long and rich history. In 1840, Steiner~\cite{steiner1840bestimmung} observed that the volume of the $\varepsilon$-neighbourhood of a convex body in $\R^3$ is a cubic polynomial in $\varepsilon$. A celebrated generalisation is due to Weyl~\cite{weyl1939volume}, who showed that the volume of the $\varepsilon$-tubular neighbourhood ($\varepsilon$ small enough) of a compact Riemannian submanifold of $\R^n$ is a polynomial in $\varepsilon$ whose coefficients are intrinsic curvature invariants; see also Hotelling~\cite{hotelling1939tubes} and the monograph by Gray~\cite{gray2004tubes}. Tube formulae entered numerical analysis through the work of Smale, Kostlan, Renegar, and Demmel, among others, who were interested in the probabilistic analysis of condition numbers (see the references in~\cite{burgisser2013condition,lotz2015volume,basu2021hausdorff}). The key observation is that if the set of ill-posed inputs of a numerical problem can be described as a subset of an algebraic variety, then a bound on the volume of its tubular neighbourhood directly translates into a bound on the probability distribution of the condition number.
Bounds on tubular neighbourhoods have been extended to singular algebraic sets by Basu and Lerario~\cite{basu2021hausdorff}.
In a complementary direction, Zhang and Kileel~\cite{zhang2025covering} bound the covering numbers of real algebraic varieties, images of polynomial maps, and general semialgebraic sets via a slicing argument that counts the connected components of generic affine sections, and deduce volume bounds for tubular neighbourhoods of such sets without any smoothness assumptions. Their bounds capture the correct leading order in $\varepsilon$, but are obtained from coverings by balls and therefore carry dimension-exponential constants rather than the degree-graded coefficients of the Weyl-type tube formula used here; the component counts of generic affine sections that control their estimates play a role analogous to the section degrees of the Gauss map that govern our tube formula, and admit Khovanskii-type bounds in the Pfaffian setting (Section~\ref{sec:pfaffian}). Among their applications are generalization bounds for deep networks with rational or ReLU activations, which measure the complexity of a network class in function space and are thereby different from our robustness analysis, which concerns the geometry of the decision boundary in input space.

Neural networks have been studied in the context of Pfaffian functions since the work of Macintyre and Sontag~\cite{macintyre1993finiteness}, who established finiteness results for sigmoidal networks. Karpinski and Macintyre~\cite{karpinski1997polynomial} showed that the VC dimension of sigmoidal Pfaffian networks is polynomial in the number of parameters; this was recently extended by D'Inverno, Bianchini, and Scarselli~\cite{dinverno2024vc} to graph neural networks with general Pfaffian activations. Bianchini and Scarselli~\cite{bianchini2014complexity} bound the sums of Betti numbers of decision regions for networks with Pfaffian activation functions. In particular, their Theorem~5 constructs, for every $\ell\ge1$, a sigmoid network with $\ell$ hidden layers and two neurons per hidden layer whose nonnegative decision region $\{x\colon f(x)\ge0\}$ has at least $2^{\ell-1}+1$ connected components. Their proof in Appendix~F iterates a fixed scalar map realised by one hidden sigmoid layer, establishing exponential growth of decision-region topology with depth even at width two.

The robustness of neural network classifiers with respect to adversarial perturbations has been of practical concern in contexts ranging from spam filtering to medical diagnosis, and a systematic investigation was initiated by Szegedy et al.~\cite{szegedy2013intriguing}. Subsequent work developed efficient attacks and robustness surrogates based on first-order or local geometric approximations, including the fast-gradient method of Goodfellow, Shlens, and Szegedy~\cite{goodfellow2015explaining}, the DeepFool algorithm of Moosavi-Dezfooli, Fawzi, and Frossard~\cite{moosavi2016deepfool}, and the analysis of adversarial, random, and semi-random perturbations by Fawzi, Moosavi-Dezfooli, and Frossard~\cite{fawzi2016robustness}. Another theoretical line explains adversarial vulnerability through high-dimensional geometry and concentration of measure; see, for example, Fawzi, Fawzi, and Fawzi~\cite{fawzi2018adversarial} and Mahloujifar, Diochnos, and Mahmoody~\cite{mahloujifar2019curse}. The robustness of classifiers is an inherently geometric problem. A classifier $F\colon \R^n\to \R^m$ partitions the input space into decision regions separated by a decision boundary $\Sigma$. Under the pairwise regularity assumptions used here, the distance $\dist(x,\Sigma)$ is the infimum of perturbation sizes that change the predicted label of an input~$x$ with a unique maximal score, and one can define a condition number $\mC(x) = \|x\|/\dist(x,\Sigma)$ by analogy with the theory of conditioning in numerical analysis~\cite{burgisser2013condition}. Bounding the volume of the tubular neighbourhood of $\Sigma$ therefore provides quantitative robustness guarantees and, more precisely, tail bounds on the probability distribution of the condition number.

In the case of neural networks with piecewise-linear activations such as ReLU, the decision boundary is itself piecewise-linear, with combinatorial complexity controlled by the linear-region structure induced across the layers. Mont\'ufar, Pascanu, Cho, and Bengio~\cite{montufar2014} initiated the systematic counting of linear regions, proving in particular exponential lower bounds for deep networks; upper bounds of order $O_{n,\ell}\!\bigl(\prod_j n_j^n\bigr)$ for an $\ell$-layer ReLU network of widths $n_1, \ldots, n_\ell$ at fixed input dimension and depth were obtained by Raghu, Poole, Kleinberg, Ganguli, and Sohl-Dickstein~\cite{raghu2017} and Serra, Tjandraatmadja, and Ramalingam~\cite{serra2018}, with refined counts by Hanin and Rolnick~\cite{hanin2019}. A different perspective due to Zhang, Naitzat, and Lim~\cite{zhang2018tropical} interprets ReLU networks as tropical rational maps and relates their linear-region structure to polyhedral subdivisions associated with a tropical-polynomial representation.

A recent program of neuroalgebraic geometry~\cite{marchetti2025neuroalgebraic} studies the function space of a network, the neuromanifold swept out as the weights vary, as a (semi-)algebraic variety, relating its algebraic invariants such as dimension, degree, and singularities to learning-theoretic properties. This shares our outlook of analysing neural networks through real-geometric invariants, including degree notions, but the object and regime differ: that program concerns the function space of polynomial networks, where algebraic geometry applies directly, whereas we study the decision boundary in input space of networks with transcendental Pfaffian activations such as the sigmoid, for which the relevant tame structure is o-minimal and the natural finiteness tools are Khovanskii's fewnomial theory and the Bernstein--Kushnirenko--Khovanskii theorem. 

\subsection{Structure of the paper}
Section~\ref{sec:pfaffian} reviews the necessary background on Pfaffian functions and Pfaffian sets in a self-contained manner. Section~\ref{sec:tubular} derives the tube formula for smooth Pfaffian hypersurfaces and applies it to obtain bounds on the volume of their tubular neighbourhoods. In Section~\ref{sec:tubular-neural}, we apply the tube formula to obtain bounds on the volume of tubular neighbourhoods of the decision boundary of neural network classifiers with Pfaffian activations. We also present our second main result, the tube formula for one-layer sigmoid networks that is polynomial in the width of the network. Section~\ref{sec:robustness} applies the tube formula to obtain tail bounds on the probability distribution of the condition number of neural network classifiers with Pfaffian activations. Section~\ref{sec:conclusions} outlines some future directions.

\subsection{Acknowledgements} The authors are grateful to Abhiram Natarajan for many insightful discussions on Pfaffian geometry and its applications. P.L. was funded by the London Mathematical Society through the Undergraduate Research Bursary URB-2022-69, and by a London School of Geometry and Number Theory--Imperial College London/King's College London/University College London PhD studentship, which is supported by the Engineering and Physical Sciences Research Council [EP/S021590/1]. M.L. was funded by EPSRC Grant EP/W00383X/1. 

\subsection{Notation and conventions}\label{subsec:notation}
We denote by $\N$ the natural numbers including $0$, by $\N_{>0}$ the positive integers, and by $[n]=\{1,\dots,n\}$, with $[0]=\varnothing$. The space $\R[X_1,\dots,X_n]_{\le d}$ consists of polynomials of total degree at most $d$. We identify a polynomial with its associated function when the domain is clear, and write $\mathcal Z(f_1,\dots,f_k)$ for their common zero set (also for nonpolynomial functions).

We use the Euclidean inner product $\langle\cdot,\cdot\rangle$ and norm $\|\cdot\|$, and write $\|x\|_\infty=\max_j|x_j|$. The vectors $e_j$ form the standard basis, $A^{\trans}$ denotes transpose, and $I_d$ is the $d\times d$ identity matrix. We write $B(p,\rho)$ for the closed Euclidean ball, $S^{n-1}(p,\rho)=\partial B(p,\rho)$, $S^{n-1}=S^{n-1}(0,1)$, and $\omega_n=\vol B(0,1)$, where $\vol$ is Euclidean volume in the ambient space. Also, $\dist(x,A)=\inf_{a\in A}\|x-a\|$, with $\dist(x,\varnothing)=\infty$.

For a smooth map $F$, $\diff{F}(x)$ denotes its differential at $x$. For a smooth scalar function $f$ and a constant vector $\tau$ in its domain space, we write
\[
\partial_\tau f(x):=\left.\frac{\diff{}}{\diff t}f(x+t\tau)\right|_{t=0}
=\diff{f}(x)[\tau]=\langle\nabla f(x),\tau\rangle.
\]
The vector $\tau$ need not be a unit vector. In coordinates, $\partial_j f=\partial_{x_j}f=\partial_{e_j}f$, $\partial_{jk}f=\partial_j\partial_k f$, and $\operatorname{Hess}f=(\partial_{jk}f)_{j,k}$. A numerical subscript refers to the corresponding coordinate of the function's current argument. A variable subscript, as in $\partial_{y_j}$ or $\partial_{t_k}$, specifies the coordinate explicitly.

We reserve $\sigma$ for the logistic sigmoid (and $\sigma^i_j$ for general hidden activations), and use $\varsigma>0$ for a Gaussian sampling scale: $\mathcal N(p,\varsigma^2 I_n)$. Subscripts on $O$ and $\Omega$ list the parameters on which their implicit constants may depend.

\section{Pfaffian functions and Pfaffian sets}\label{sec:pfaffian}
References for the material in this section are~\cite{khovanskii1991fewnomials, zell2003quantitative, gabrielov2004complexity}. We assume throughout that $n\in \N_{>0}$. Pfaffian functions are defined relative to an open domain $\cU\subseteq\R^n$; unless a domain is specified we take $\cU=\R^n$, which covers the neural network activation functions of interest.

\subsection{Pfaffian functions}
Pfaffian functions, introduced by Khovanskii~\cite{khovanskii1991fewnomials}, are functions that satisfy triangular systems of first-order partial differential equations with polynomial coefficients. For chains defined on all of $\R^n$, the sets defined by Pfaffian functions are tame in the sense of o-minimal geometry~\cite{wilkie1996model, wilkie1999theorem, speissegger1999pfaffian}, and many of their geometric and topological properties, such as the number of connected components and the sum of Betti numbers, can be bounded effectively in terms of the format~\cite{gabrielov2004complexity}.

\begin{definition}[Pfaffian function]
Let $\cU \subset \R^n$ be an open set. A \emph{Pfaffian chain} of length (also called order) $s \in \N$ and chain degree $\alpha \in \N_{>0}$ over $\cU$ is a sequence of functions $\boldq= (q_1, \dotsc, q_s)$ with $q_i\in C^\infty(\cU)$ for $i\in [s]$, such that there exist polynomials $P_{ij}\in \R[X_1,\dots,X_n,Y_1,\dots,Y_i]_{\leq \alpha}$, for $i\in [s]$ and $j\in [n]$, that verify
\begin{equation}\label{eq:pfaffchain}
\frac{\partial q_i}{\partial x_j}(x) = P_{ij}(x, q_1(x), \dotsc, q_i(x)).
\end{equation}
A function $g(x)=P(x,q_1(x),\dots,q_s(x))$, with $P\in \R[X_1,\dots,X_n,Y_1,\dots,Y_s]_{\leq \beta}$ for $\beta\in \N$, is called a \emph{Pfaffian function} of chain degree $\alpha$, degree $\beta$, and chain length $s$. A function $\cU\to \R^m$ is called Pfaffian if all its components are Pfaffian.
\end{definition}

The triple $(\alpha, \beta, s)$ is called a \emph{format} of $g$; its degree parameters are upper bounds for the chosen representation. We denote by $\Pfaff_{\boldq,\beta}(\cU)$ the set of all Pfaffian functions over $\cU$ with chain $\boldq$ and degree at most $\beta$, and by $\Pfaff_{\alpha,\beta,s}(\cU)$ the set of all Pfaffian functions over $\cU$ with format $(\alpha,\beta,s)$. When a map $\cU\to \R^m$ has a particular format, every component admits that format, and we write $\Pfaff_{\alpha,\beta,s}(\cU; \R^m)$. A common chain for the components is required only when explicitly stated.

\begin{remark}
 Note that the Pfaffian chain associated with a Pfaffian function, and hence also a format, is not unique. In particular, $\Pfaff_{\alpha,\beta,s}(\cU)\subseteq \Pfaff_{\alpha',\beta',s'}(\cU)$ for $\alpha'\geq \alpha$, $\beta'\geq \beta$, and $s'\geq s$. In the following, we will often leave the chain implicit, and our results will be stated in terms of the format.
\end{remark}

We call a Pfaffian function $g\colon \cU\to \R$ \emph{autonomous} with respect to a Pfaffian chain $\boldq$ if 
\begin{equation}\label{eq:simplepfaff}
  \frac{\partial q_i}{\partial x_j}(x) = P_{ij}(q_1(x),\dots,q_i(x))
\end{equation}
for some $P_{ij}\in \R[Y_1,\dots,Y_i]$, and $g(x)=P(q_1(x), \dots, q_s(x))$ for some $P\in \R[Y_1,\dots,Y_s]$. 
Note that being autonomous is a property of a given Pfaffian representation of a function, not a property of the function itself. Every Pfaffian function can be made autonomous by adjoining the coordinate functions $x_1,\dots,x_n$ to the Pfaffian chain. In what follows, we often work with a fixed chain that may be implicit, and refer to a Pfaffian function as autonomous if it is autonomous with respect to that chain.

\begin{example}\label{ex:1} The following are simple examples of Pfaffian functions.
        \begin{enumerate}
            \item A polynomial $P\in\R[X_1,\dots,X_n]_{\le\beta}$ gives a Pfaffian function with format $(\alpha,\beta,0)$ for any $\alpha>0$;
            \item $\exp\colon \R\to \R$ is Pfaffian with format $(1,1,1)$. A Pfaffian chain is $\boldq=(\exp)$ and $P_{11}(X,Y_1)=Y_1$;
            \item $\tanh$ is Pfaffian with format $(2,1,1)$. A Pfaffian chain is $\boldq=(\tanh)$ and $P_{11}(X,Y_1)=1-Y_1^2$;
            \item The logistic sigmoid $\sigma(x) = (1 + \econst^{-x})^{-1}$ is Pfaffian with format $(2, 1, 1)$. A Pfaffian chain is $\boldq=(\sigma)$ and $P_{11}(X,Y_1)=Y_1(1-Y_1)$;
            \item $\arctan$ is Pfaffian with format $(3, 1, 2)$. A Pfaffian chain is $\boldq=((1+x^2)^{-1},\arctan(x))$. The polynomials are $P_{11}(X,Y_1) = -2XY_1^2$ and $P_{21}(X,Y_1,Y_2)=Y_1$. 
        \end{enumerate}
Examples (2--4) are autonomous with respect to the given chains, whereas nonconstant polynomials in (1) and the arctangent representation in (5) are not.
\end{example}



Pfaffian functions enjoy some convenient closure properties under algebraic operations and composition, as illustrated in the following two results (see also~\cite[Proposition 1.8]{zell2003quantitative}).

\begin{lemma}\label{le:pfaffalgebra}
Let $g\in \Pfaff_{\alpha, \beta, s}(\cU)$ and $h\in \Pfaff_{\alpha', \beta', s'}(\cU)$, with $\beta,\beta'\ge1$. Then:
\begin{enumerate}
\item $g + h \in \Pfaff_{\max\{\alpha,\alpha'\}, \max\{\beta,\beta'\},s+s'}(\cU)$;
\item $gh \in \Pfaff_{\max\{\alpha,\alpha'\}, \beta+\beta',s+s'}(\cU)$;
\item For each $i\in [n]$, $\partial_i g \in \Pfaff_{\alpha,\alpha+\beta-1,s}(\cU)$, with respect to the same chain as $g$.
\end{enumerate}
\end{lemma}

\begin{proof}
The first two statements are straightforward consequences of the fact that the concatenation of Pfaffian chains is again a Pfaffian chain. For the third statement, assume $g(x)=P(x,q_1(x),\dots,q_s(x))$ for a polynomial $P\in \R[X_1,\dots,X_n,Y_1,\dots,Y_s]_{\leq \beta}$. Then
\begin{align*}
\partial_i g(x)={}&(\partial_{X_i}P)(x,\boldq(x))\\
&+\sum_{\ell=1}^s(\partial_{Y_\ell}P)(x,\boldq(x))
P_{\ell i}(x,q_1(x),\dots,q_\ell(x)).
\end{align*}
The resulting expression is a polynomial in $x$ and the $q_j$, with degree bounded by $\alpha+\beta-1$.
\end{proof}

\begin{remark}\label{re:1}
The chain-length bounds for $g+h$ and $gh$ are in general not sharp: if $g$ and $h$ share a chain $\boldq$ of length $s$, then their sum and product use that same chain. Likewise, for a constant vector $\tau$, $\partial_\tau g=\sum_i\tau_i\partial_i g$ has format $(\alpha,\alpha+\beta-1,s)$ on the chain of $g$. More generally, a linear combination of Pfaffian functions $g_1,\dots,g_m$ with formats $(\alpha_i,\beta_i,s_i)$ has format $(\max_i\{\alpha_i\},\max_i\{\beta_i\},\sum_i s_i)$.
\end{remark}

\begin{lemma}\label{le:composition}
Let $g\in \Pfaff_{\alpha,\beta,s}(\cU;\R^m)$ and $h\in \Pfaff_{\alpha',\beta',s'}(\R^m)$ be Pfaffian functions, and assume $\beta,\beta'\ge1$ and $s'\geq 1$.
\begin{enumerate}
\item $h\circ g \in \Pfaff_{(\alpha'+1)\beta+\alpha-1, \beta\beta', ms+s'}(\cU)$;
\item If $h$ is autonomous, then $h\circ g \in \Pfaff_{\alpha'+\alpha+\beta-1,\beta', ms+s'}(\cU)$;
\item If all components of $g$ share a Pfaffian chain $\boldq$ of length $s$, then $h\circ g$ admits a chain of length $s+s'$, with the degree bounds in (1) or (2), respectively.
\end{enumerate}
\end{lemma}

\begin{proof}
Let $g=(g_1,\dots,g_m)$, and for each $i\in [m]$, let $\boldq^{(i)}=(q_1^{(i)},\dots,q_s^{(i)})$ be a Pfaffian chain for $g_i$. Let $\boldq'=(q_1',\dots,q_{s'}')$ be a Pfaffian chain for $h$. Let $P$ and $P_i$, $i\in [m]$, be the polynomials representing $h$ and the $g_i$, respectively. The composition $h\circ g$ is then represented as 
\begin{equation}\label{eq:longpoly}
 (h\circ g)(x) = P(P_1(x,\boldq^{(1)}(x)),\dots,P_m(x,\boldq^{(m)}(x)),\boldq'(g(x))),
\end{equation}
which is clearly a polynomial in $x=(x_1,\dots,x_n)$ and in the chain
\begin{equation}\label{eq:newchain}
  (q_1^{(1)},\dots,q_s^{(1)},\dots,q_1^{(m)},\dots,q_s^{(m)},q'_1\circ g,\dots,q_{s'}'\circ g)
\end{equation}
of length $ms+s'$. It remains to be seen that this is a Pfaffian chain. For the first $ms$ functions in this chain there is nothing to show. It remains to be seen that the derivatives of $q_i'\circ g$ can be expressed as polynomials in $x$ and in the previous elements of the chain. In fact, for $j\in [n]$ we have
\begin{equation}\label{eq:composite}
\begin{aligned}
\partial_j(q'_i\circ g)(x)
&=\sum_{\ell=1}^m(\partial_\ell q'_i)(g(x))\,\partial_j g_\ell(x)\\
&=\sum_{\ell=1}^m P'_{i\ell}(g(x),q'_1(g(x)),\dots,q'_i(g(x)))\,\partial_j g_\ell(x),
\end{aligned}
\end{equation}
where the $P_{i\ell}'$ are the polynomials associated to the chain $\boldq'$. 
This expresses the partial derivative as a polynomial in the coordinates $x$ and the chain elements in~\eqref{eq:newchain} up to $q_i'\circ g$. Indeed, $g_{\ell}$ and, by Lemma~\ref{le:pfaffalgebra}(3), its partial derivatives are polynomials in $x$ and $q_1^{(\ell)},\dots,q_s^{(\ell)}$. This establishes that~\eqref{eq:newchain} is a Pfaffian chain of length $ms+s'$.

To prove the degree bounds in the general case (1), note that the degree of the composition of polynomials is bounded by the product of the degrees. The expression $P_{i\ell}'(g(x),q'_1(g(x)),\dots,q'_i(g(x)))$ has degree at most $\alpha'\beta$ in $x$ and the new chain elements. Lemma~\ref{le:pfaffalgebra}(3) bounds the degree of $\partial_j g_\ell$ by $\alpha+\beta-1$, establishing the bound $(\alpha'+1)\beta+\alpha-1$ on the degree of the new Pfaffian chain~\eqref{eq:newchain}.
Moreover, the degree of the polynomial~\eqref{eq:longpoly} representing $h\circ g$ is clearly bounded by $\beta\beta'$. 
For the autonomous case (2), the polynomials $P'_{i\ell}$ have no explicit dependence on the input coordinates. Thus $P'_{i\ell}(q'_1(g(x)),\dots,q'_i(g(x)))$ has degree at most $\alpha'$ in the last $s'$ chain elements, while $\partial_j g_\ell$ has degree at most $\alpha+\beta-1$ in the remaining variables. This bounds the degree of~\eqref{eq:composite} by $\alpha'+\alpha+\beta-1$. We also observe from~\eqref{eq:longpoly} that in the autonomous case, $h\circ g$ depends only on the last $s'$ elements of its Pfaffian chain in the same way that $h$ depends on its Pfaffian chain $\boldq'$. This implies that the degree is bounded by $\beta'$.
Case (3) is clear.
\end{proof}

\begin{remark}\label{re:affine}
Pullback by an affine map preserves a Pfaffian format: substituting affine expressions in~\eqref{eq:pfaffchain} does not increase the polynomial degrees. In particular, a rigid change of coordinates and restriction to an affine subspace preserve the format. This also covers polynomial functions represented by an empty chain.
\end{remark}

We are also interested in the geometric objects defined by Pfaffian functions.

\begin{definition}
A \emph{Pfaffian set} (or Pfaffian variety) is a set of the form 
\begin{equation*}
  V = \mathcal{Z}(g_1, \dotsc, g_k) = \{x\in \R^n\colon g_1(x)=\cdots = g_k(x)=0\},
\end{equation*}  
where the $g_i\colon \R^n\to \R$ are Pfaffian functions. A {\em basic semi-Pfaffian set} is a set of the form
\begin{equation*}
 B = \{x\in \R^n \colon g_1(x)=\cdots = g_k(x)=0, h_1(x)>0,\dots,h_{\ell}(x)>0\},
\end{equation*}
where the $g_i$ and $h_j$ are Pfaffian functions. A {\em semi-Pfaffian set} is a finite union of basic semi-Pfaffian sets.
\end{definition}

Note that semi-Pfaffian sets are precisely the sets that can be written as unions, intersections, and complements of sets defined by expressions of the form $g_i(x) \star 0$, with $\star \in \{=,\, \leq,\, \geq\}$. 

\begin{remark}
  Since the concatenation of Pfaffian chains is again a Pfaffian chain, we can assume without loss of generality that the functions $g_i$ and $h_j$ in the definition of a (semi-)Pfaffian set are Pfaffian with respect to the same Pfaffian chain.
\end{remark} 

We conclude this subsection with the following central result on Pfaffian systems.
For a smooth map $F\colon \R^n\to\R^k$, a value $y\in\R^k$ is \emph{regular} if $\diff{F}(x)$ is surjective at every $x\in F^{-1}(y)$. A solution of $F(x)=0$ is regular if the differential there has rank $k$; for a square system ($k=n$), we also call such a solution \emph{non-degenerate}.

\begin{theorem}[Khovanskii]\label{thm:khovanskii}
Let $f_1,\dots,f_n\colon\R^n\to\R$ be Pfaffian functions with respect to a common chain of length $s\ge0$ and chain degree $\alpha\ge1$, with component degrees bounded by $\beta_1,\dots,\beta_n\ge1$. Then the number of non-degenerate real solutions of $f_1=\cdots=f_n=0$ is at most
\begin{equation*}
2^{\frac{s(s-1)}2}\beta_1\cdots\beta_n
\left(\beta_1+\cdots+\beta_n-n+\min\{n,s\}\alpha+1\right)^s.
\end{equation*}
\end{theorem}

We use the global version of the theorem, as stated in~\cite[Theorem 3.1]{gabrielov2004complexity}; see also~\cite[\S 3.12, Corollary 5]{khovanskii1991fewnomials}. For general open domains, a bound must also account for the domain's complexity~\cite[Remark 2.12]{gabrielov2004complexity}. Every application below is on all of a Euclidean space. Constant functions can be assigned a positive degree bound when applying this theorem.
A qualitative derivation of finiteness appears in~\cite{Marker1997}, while a self-contained and complete proof appears in~\cite{lotz2026khovanskii}; for parameterwise sharpness of the quantitative bound, see~\cite{bickerton2026sharpnesskhovanskiisbezouttypebound}.

The following component bound follows by the usual compact perturbation and Morse-theoretic argument~\cite[Corollary 3.3]{gabrielov2004complexity}, keeping the $\min\{n,s\}$ term in Theorem~\ref{thm:khovanskii}; see also~\cite{lotz2026khovanskii}.

\begin{corollary}\label{cor:pfaffian-components}
Let $F=(f_1,\dots,f_m)\colon\R^n\to\R^m$ have a common Pfaffian chain of length $s$, chain degree $\alpha\ge1$, and component degrees at most $\beta\ge1$. Then the number of connected components of $\mathcal Z(F)$ is at most
\begin{equation}\label{eq:jones}
2^{\frac{s(s-1)}2+1}\beta(\alpha+2\beta-1)^{n-1}
\left[n(\alpha+2\beta-2)+\alpha(\min\{n,s\}-1)+2\right]^s.
\end{equation}
\end{corollary}

\begin{proof}
As in~\cite[Corollary 3.3]{gabrielov2004complexity}, choose $R$ sufficiently large and then $\eta>0$ sufficiently small so that
\[
F_\eta(x):=\sum_{j=1}^m f_j(x)^2+\eta(\|x\|^2-R^2)
\]
has a smooth compact zero set with at least as many connected components as $\mathcal Z(F)$. A generic linear height function on this hypersurface is Morse and has a critical point on every component. After a rotation, its critical points solve
\[
F_\eta=\partial_2 F_\eta=\cdots=\partial_n F_\eta=0.
\]
All these solutions are non-degenerate. Their degrees are bounded by $2\beta$ and $\alpha+2\beta-1$ for the remaining $n-1$ equations. Substituting these degrees in Theorem~\ref{thm:khovanskii} gives~\eqref{eq:jones}.
\end{proof}

\section{Tubular neighbourhoods of Pfaffian hypersurfaces}\label{sec:tubular}
For $A\subseteq\R^n$ and $\varepsilon>0$, write
\[
T(A,\varepsilon)=\{z\in\R^n\colon\dist(z,A)\le\varepsilon\}
\]
for the full closed distance neighbourhood. If $M$ is a smooth $m$-dimensional submanifold of $\R^n$, with the induced Euclidean metric and positive codimension $c=n-m$, define its \emph{normal tube} by
\[
T^\perp(M,\varepsilon)
=\{x+u\colon x\in M,\ u\perp T_xM,\ \|u\|\le\varepsilon\}.
\]
Always $T^\perp(M,\varepsilon)\subseteq T(M,\varepsilon)$. If $M$ is closed as a subset of $\R^n$ and has no boundary, equality holds for every $\varepsilon>0$: a nearest point exists and the displacement to it is normal to $M$. Neither boundedness nor uniqueness of the nearest point is needed. For manifolds with boundary, the two neighbourhoods can differ.

We combine a tube-volume estimate from~\cite{lotz2015volume} with effective bounds for the generalised Gauss map. A localisation argument will allow the hypersurface itself to be unbounded.

\subsection{Generalised Gauss map}
For $x\in M$, write $N_xM=(T_xM)^\perp$ and
\[
S(NM)=\{(x,v)\in M\times S^{n-1}\colon v\in N_xM\}
\]
for the normal space and normal sphere bundle, respectively. The latter is a smooth manifold of dimension $n-1$.

\begin{definition}
The \emph{generalised Gauss map} is
\[
\gamma_M\colon S(NM)\longrightarrow S^{n-1},\qquad (x,v)\longmapsto v.
\]
\end{definition}

For a smooth map $\psi$ between manifolds of the same dimension, define its \emph{maximal degree} by
\[
\md(\psi):=\sup_{v\in\reg(\psi)}|\psi^{-1}(v)|\in\N\cup\{\infty\},
\]
where $\reg(\psi)$ is the set of regular values and $|\psi^{-1}(v)|$ counts distinct points, without signs or multiplicities. Each regular fibre is discrete, and is finite when the domain is compact~\cite[\S 1]{milnor1997topology}. A uniform bound on these cardinalities requires further information; for the Pfaffian manifolds below it follows from effective zero counts.

Let $\cE_j^n$ denote the space of $j$-dimensional affine subspaces of $\R^n$. Almost every $H\in\cE_{c+i}^n$ meets $M$ transversally, and then $M\cap H$ is an $i$-dimensional submanifold of $H\cong\R^{c+i}$. Define
\[
\md(M,H):=\md(\gamma_{M\cap H}),
\qquad
\md_i(M):=\sup_{H\in\cE_{c+i}^n}\md(M,H),
\quad 0\le i\le m,
\]
with $\md(M,H)=0$ for a non-transverse intersection or an empty section. The normal spaces of $M\cap H$ are taken within $H$. We abbreviate $\md(M):=\md_m(M)=\md(\gamma_M)$.

\begin{lemma}\label{lem:degree-restriction}
If $U$ is an open subset of $M$, then $\md_i(U)\le\md_i(M)$ for every $0\le i\le m$.
\end{lemma}

\begin{proof}
Fix a transverse section of $U$, a regular value of its Gauss map, and finitely many points of that fibre. Transversality and the inverse function theorem imply that these distinct regular preimages persist under sufficiently small changes of the affine section and normal direction. Choose a nearby section transverse to all of $M$, and then a nearby direction regular for its Gauss map; such choices exist by transversality and Sard's theorem. The chosen finite collection has cardinality at most $\md_i(M)$. Taking suprema over finite collections, regular values and sections proves the claim.
\end{proof}

The following normal-tube estimate follows from~\cite[Theorems 3.1 and 4.3]{lotz2015volume}, with the boundary convention explained below.

\begin{theorem}[Normal-tube bound]\label{thm:tube_bound}
Let $M\subset B(p,\rho)$ be a compact smooth submanifold of dimension $0\le m<n$, possibly with boundary, and set $c=n-m$. Then
\[
\vol T^\perp(M,\varepsilon)\le
2\omega_n\rho^n\sum_{i=0}^m\binom n{c+i}\md_i(M^\circ)
\left(\frac{\varepsilon}{\rho}\right)^{c+i},
\]
where $M^\circ=M\setminus\partial M$ and $\varepsilon,\rho>0$.
\end{theorem}

\begin{remark}\label{rem:tube_bound_boundary}
Theorem 3.1 of~\cite{lotz2015volume} bounds the normal tube based on an open subset $U\subseteq M\setminus\partial M$. Apply it with $U=M^\circ$ and use the section-degree argument of its Theorem 4.3; the extension of the Gauss-map estimate to such $U$ is stated after its Lemma 4.2. This gives the displayed bound for $\vol T^\perp(M^\circ,\varepsilon)$. Normal vectors of length at most $\varepsilon$ based on $\partial M$, with normal taken relative to $M$, form a disk bundle of dimension $(m-1)+c=n-1$. Its image has zero $n$-dimensional volume, so $\vol T^\perp(M,\varepsilon)=\vol T^\perp(M^\circ,\varepsilon)$.

This does not bound the full distance neighbourhood $T(M,\varepsilon)$ when $M$ has boundary. For example, the normal tube of a line segment in $\R^2$ is a rectangle, whereas its distance neighbourhood also includes two semicircular caps. The tube $T^\perp(\partial M,\varepsilon)$ uses normals to $\partial M$ as a separate submanifold and can have positive $n$-dimensional volume.
\end{remark}

\subsection{Localisation and Gauss-map degree bounds}
The next observation avoids introducing a separate tube around the boundary of a truncation.

\begin{proposition}\label{prop:local-tube}
Let $V\subset\R^n$ be a closed smooth hypersurface without boundary, with finite degrees $\md_i(V)$. If $X$ is uniform in $B(p,\rho)$, where $\rho,\varepsilon>0$, and $u=\varepsilon/\rho$, then
\begin{equation}\label{eq:local-tube}
\prob\{\dist(X,V)\le\varepsilon\}\le
2\sum_{i=0}^{n-1}\binom n{i+1}\md_i(V) u^{i+1}(1+u)^{n-i-1}.
\end{equation}
\end{proposition}

\begin{proof}
The empty case is immediate. Choose $R>\rho+\varepsilon$ such that the sphere $S^{n-1}(p,R)$ meets $V$ transversally. Then $M_R=V\cap B(p,R)$ is a compact manifold with boundary. If $z\in B(p,\rho)$ has distance at most $\varepsilon$ from $V$, choose a nearest point $x\in V$. We have
\[
\|x-p\|\le\|x-z\|+\|z-p\|\le\varepsilon+\rho<R,
\]
so $x\in M_R^\circ$. The nearest-point condition gives $z-x\perp T_xV$, and hence
\[
T(V,\varepsilon)\cap B(p,\rho)
\subseteq T^\perp(M_R^\circ,\varepsilon)
\subseteq T^\perp(M_R,\varepsilon).
\]
Apply Theorem~\ref{thm:tube_bound} to the compact manifold $M_R$, and Lemma~\ref{lem:degree-restriction} to its interior, which is open in $V$. Dividing by $\omega_n\rho^n$ bounds the probability by
\[
2\sum_{i=0}^{n-1}\binom n{i+1}\md_i(V)
u^{i+1}\left(\frac R\rho\right)^{n-i-1}.
\]
By Sard's theorem, regular radii can be chosen decreasing to $\rho+\varepsilon$. Taking this limit proves~\eqref{eq:local-tube}.
\end{proof}

For hypersurfaces, a Gauss fibre is described by a square system without auxiliary variables. The following non-degeneracy statement will be used in both the Pfaffian and Bernstein bounds. In what follows, the notation $\partial_{\tau}f=\langle\nabla f,\tau\rangle$ is used for the directional derivative.

\begin{lemma}\label{le:gauss-nondegenerate}
Let $f:\R^d\to\R$ be smooth with $\nabla f\ne0$ on $V=\mathcal Z(f)$, and let $v\in S^{d-1}$ be a regular value of $\gamma_V$. Choose any basis $\tau_1,\dots,\tau_{d-1}$ of $v^\perp$. Then the solutions of
\begin{equation}\label{eq:gauss-directional-system}
f=\partial_{\tau_1}f=\cdots=\partial_{\tau_{d-1}}f=0
\end{equation}
are in bijection with $\gamma_V^{-1}(v)$, and every solution is non-degenerate. In particular, when $v_d\ne0$, choosing $\tau_j=v_de_j-v_je_d$ gives
\begin{equation}\label{eq:gauss-cross-system}
f=0,\qquad v_d\partial_j f-v_j\partial_d f=0,\quad 1\le j<d.
\end{equation}
\end{lemma}

\begin{proof}
The derivative equations say exactly that the nonzero gradient of $f$ is parallel to $v$. Each such point $x$ contributes precisely the bundle point $(x,v)$ to the specified fibre.

A rotation and an invertible change of the derivative equations reduce non-degeneracy to $v=e_d$ and $\tau_j=e_j$. At a solution $x$, we have $\partial_d f(x)\ne0$, and the implicit function theorem represents $V$ locally as $u\mapsto(u,h(u))$. Write $x=(u_0,h(u_0))$; then $\nabla h(u_0)=0$. Twice differentiating $f(u,h(u))=0$ gives
\[
\partial_{jk}f(x)=-\partial_d f(x)\,\partial_{jk}h(u_0)
\quad(1\le j,k<d).
\]
Regularity of the Gauss map is equivalent to nonsingularity of $\operatorname{Hess}h(u_0)$. Expanding the Jacobian determinant of $(f,\partial_1f,\dots,\partial_{d-1}f)$ along its first row gives absolute value
\[
|\partial_d f(x)|^d\,|\det\operatorname{Hess}h(u_0)|>0.
\]
For $d=1$, the system is simply $f=0$ and the assertion follows from $f'(x)\ne0$.
\end{proof}

\begin{proposition}\label{prop:hypersurface-degree}
Let $f\in\Pfaff_{\alpha,\beta,s}(\R^n)$, where $\alpha,\beta\ge1$, and assume that $0$ is a regular value. For $V=\mathcal Z(f)$ and $0\le i<n$,
\begin{equation}\label{eq:hypersurface-ideg}
\md_i(V)\le
2^{\frac{s(s-1)}2}\beta(\alpha+\beta-1)^i
\left[\beta+i(\alpha+\beta-2)+\min\{i+1,s\}\alpha\right]^s.
\end{equation}
\end{proposition}

\begin{proof}
Parametrise a transverse affine section $H\in\cE_{i+1}^n$ by an affine isometry from $\R^{i+1}$. By Remark~\ref{re:affine}, the restricted function has the same format and $0$ remains a regular value. For each regular Gauss direction, Lemma~\ref{le:gauss-nondegenerate} supplies a non-degenerate square system with one equation of degree $\beta$ and $i$ equations of degree at most $\alpha+\beta-1$, by Lemma~\ref{le:pfaffalgebra}(3). Theorem~\ref{thm:khovanskii} gives~\eqref{eq:hypersurface-ideg}. The bound holds for every such fibre, so taking the two suprema proves the result without a compactness assumption.
\end{proof}

For completeness, a similar argument with Lagrange multipliers gives a bound for complete intersections.

\begin{proposition}\label{prop:pfaffian-degree-bound}
Let $M=\mathcal Z(f_1,\dots,f_c)\subset\R^n$, where $1\le c\le n$, the gradients $\nabla f_j$ are linearly independent on $M$, and the functions share a Pfaffian chain with formats $(\alpha,\beta_j,s)$, $\alpha,\beta_j\ge1$. Put $\overline\beta=\max_j\beta_j$. Then
\begin{align*}
\md(M)\le{}&
2^{\frac{s(s-1)}2}(\alpha+\overline\beta)^n\prod_{j=1}^c\beta_j\\
&{}\cdot\left[
n(\alpha+\overline\beta-1)+\sum_{j=1}^c(\beta_j-1)
+\min\{n+c,s\}\alpha+1
\right]^s.
\end{align*}
The same bound in ambient dimension $c+i$ bounds $\md_i(M)$.
\end{proposition}

\begin{proof}
For a regular value $v$ of $\gamma_M$, its fibre is in bijection with solutions of
\begin{equation}\label{eq:pfaffian-system}
f_j(x)=0\quad(1\le j\le c),\qquad
\sum_{j=1}^c\lambda_j\nabla f_j(x)=v.
\end{equation}
The multipliers are unique by independence of the gradients. To check non-degeneracy, identify $(x,\lambda)\in M\times\R^c$ with $(x,\sum_j\lambda_j\nabla f_j(x))$ in the normal bundle. Near a solution, projection of this bundle to $\R^n$ is, in radial coordinates, $(x,u,r)\mapsto r\gamma_M(x,u)$. Its differential is invertible at $r=1$ precisely when $\diff\gamma_M$ is invertible. Since $(f_1,\dots,f_c)$ is a submersion on $M$, this is equivalent to the full Jacobian of~\eqref{eq:pfaffian-system} being nonsingular. Theorem~\ref{thm:khovanskii} now applies in $n+c$ variables, with degrees $\beta_1,\dots,\beta_c$ and $n$ degrees at most $\alpha+\overline\beta$. Restriction to a transverse affine section preserves the formats.
\end{proof}

\subsection{A Pfaffian tube formula}
Combining the local tube bound with the direct hypersurface system gives the following estimate.

\begin{theorem}\label{thm:prob_bound}
Let $f\in\Pfaff_{\alpha,\beta,s}(\R^n)$ with $\alpha,\beta\ge1$ and $0$ a regular value, and set $V=\mathcal Z(f)$. Let $X$ be uniformly distributed in $B(p,\rho)$, where $p\in\R^n$ and $\rho>0$. For every $\varepsilon>0$,
\[
\prob\{\dist(X,V)\le\varepsilon\}\le
C_{\alpha,\beta,s,n}
\left[\left(1+(\alpha+\beta)\frac{\varepsilon}{\rho}\right)^n
-\left(1+\frac{\varepsilon}{\rho}\right)^n\right],
\]
where
\begin{equation}\label{eq:pfaffian-tube-constant}
C_{\alpha,\beta,s,n}
=\frac{2^{\frac{s(s-1)}2+1}\beta}{\alpha+\beta-1}
\left[\beta+(n-1)(\alpha+\beta-2)+\min\{n,s\}\alpha\right]^s.
\end{equation}
\end{theorem}

\begin{proof}
Put $D=\alpha+\beta-1$, $B=\beta+(n-1)(\alpha+\beta-2)+\min\{n,s\}\alpha$, and $A=2^{s(s-1)/2}\beta B^s$. Proposition~\ref{prop:hypersurface-degree} gives $\md_i(V)\le AD^i$. The zero set $V$ is closed, so Proposition~\ref{prop:local-tube} applies. With $u=\varepsilon/\rho$, the binomial identity
\[
\sum_{i=0}^{n-1}\binom n{i+1}D^i u^{i+1}(1+u)^{n-i-1}
=\frac{(1+(D+1)u)^n-(1+u)^n}{D}
\]
gives the result, since $C_{\alpha,\beta,s,n}=2A/D$.
\end{proof}

\begin{remark}\label{rem:prob_bound_simple}
The right-hand side vanishes with $\varepsilon$ and has expansion
\[
C_{\alpha,\beta,s,n}\,n(\alpha+\beta-1)\frac{\varepsilon}{\rho}
+O_{\alpha,\beta,s,n}\!\left(\frac{\varepsilon^2}{\rho^2}\right)
\quad\text{as }\varepsilon/\rho\to0.
\]
This linear leading term yields the $O(1/t)$ condition-number tails in Section~\ref{sec:robustness}. Every probability bound may also be replaced by its minimum with $1$.
\end{remark}

\section{Tubular neighbourhoods of neural networks}\label{sec:tubular-neural}
Neural networks with many common smooth activation functions are Pfaffian~\cite{karpinski1997polynomial,bianchini2014complexity}. When used as classifiers, they define semi-Pfaffian decision regions, and Theorem~\ref{thm:prob_bound} bounds the probability of lying close to a regular pairwise equality hypersurface. For a single hidden layer with logistic sigmoid activation and rational first-layer weights, we obtain bounds polynomial in the width.

\subsection{Neural networks as Pfaffian functions}
Let $F=F^{\ell+1}\circ\cdots\circ F^1:\R^n\to\R^m$ be a network with $\ell\ge1$ hidden layers, widths $n_1,\dots,n_\ell$, and $h=\sum_{i=1}^\ell n_i$ hidden units. Write $n_0=n$ and
\[
F^i(z)=\sigma^i(A^iz+b^i)\quad(1\le i\le\ell),
\qquad F^{\ell+1}(z)=g(z),
\]
where $A^i\in\R^{n_i\times n_{i-1}}$ and $b^i\in\R^{n_i}$. Each activation acts coordinatewise:
\[
\sigma^i(z_1,\dots,z_{n_i})
=(\sigma^i_1(z_1),\dots,\sigma^i_{n_i}(z_{n_i})),
\qquad \sigma^i_j:\R\to\R.
\]
The output map $g:\R^{n_\ell}\to\R^m$ may be affine, or an affine map followed by softmax.

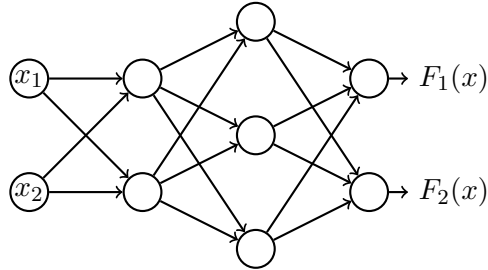
\begin{figure}[htbp]
\centering
\begin{tikzpicture}[thick, scale=1.5]
  \node [every neuron] (I1) at (-2,0.5) {};
  \node [every neuron] (I2) at (-2,-0.5) {};
  \node [every neuron] (M1) at (-1,1) {};
  \node [every neuron] (M2) at (-1,0) {};
  \node [every neuron] (M3) at (-1,-1) {};
  \node [every neuron] (O1) at (0,0.5) {};
  \node [every neuron] (O2) at (0,-0.5) {};
  
  \node (OO1) at (0.75,0.5) [] {$F_1(x)$};
  \node (OO2) at (0.75,-0.5) [] {$F_2(x)$};
  \node (II1) at (-3,0.5) [] {$x_1$};
  \node (II2) at (-3,-0.5) [] {$x_2$};
  \node [every neuron] (II1) at (-3,0.5) {};
  \node [every neuron] (II2) at (-3,-0.5) {};
  
  \draw[->] (II1) -- (I1);
  \draw[->] (II1) -- (I2);
  \draw[->] (II2) -- (I1);
  \draw[->] (II2) -- (I2);
  \draw[->] (I1) -- (M1);
  \draw[->] (I1) -- (M2);
  \draw[->] (I2) -- (M2);
  \draw[->] (I2) -- (M3);
  \draw[->] (I1) -- (M3);
  \draw[->] (I2) -- (M1);
  \draw[->] (M1) -- (O1);
  \draw[->] (M2) -- (O1);
  \draw[->] (M2) -- (O2);
  \draw[->] (M3) -- (O2);
  \draw[->] (M1) -- (O2);
  \draw[->] (M3) -- (O1);
  \draw[->] (O1) -- (OO1);
  \draw[->] (O2) -- (OO2);
\end{tikzpicture}
\caption{A fully connected neural network with $n = 2, m=2, \ell = 2, h = 5$}\label{fig:nn}
\end{figure}

\begin{proposition}\label{prop:nn-format}
Assume that every hidden activation has an autonomous Pfaffian representation with format $(\alpha,\beta,s)$, where $\alpha,\beta,s\ge1$. For $1\le i\le\ell$, the components of $F^i\circ\cdots\circ F^1$ share a chain with format
\[
\left(i(\alpha+\beta-1)-\beta+1,\ \beta,\ s\sum_{j=1}^i n_j\right).
\]
If $g$ is affine, all components of $F$ and all their pairwise differences share the format
\begin{equation}\label{eq:nn-affine-format}
\left(\ell(\alpha+\beta-1)-\beta+1,\ \beta,\ sh\right).
\end{equation}
If instead each component of $g$ has an autonomous representation with format $(\alpha',\beta',s')$, where $\alpha',\beta',s'\ge1$, concatenating these output chains gives a common format
\begin{equation}\label{eq:nn-autonomous-format}
\left(\ell(\alpha+\beta-1)+\alpha',\ \beta',\ sh+s'm\right).
\end{equation}
\end{proposition}

\begin{proof}
Each hidden unit contributes its activation chain evaluated at its scalar preactivation. Order these chains by layer and concatenate within each layer. For the first layer the chain degree is $\alpha_1=\alpha$ and the component degree is $\beta$. If the preceding layers have chain degree $\alpha_{i-1}$, the derivatives of each preactivation have degree at most $\alpha_{i-1}+\beta-1$. The chain rule for an autonomous activation therefore gives
\[
\alpha_i=\alpha_{i-1}+\alpha+\beta-1
=i(\alpha+\beta-1)-\beta+1.
\]
The component degree remains $\beta$, and layer $i$ adds $sn_i$ chain elements. An affine output changes neither the chain nor the degree bound. For autonomous outputs, Lemma~\ref{le:composition}(2)--(3), followed by concatenation of their chains, gives~\eqref{eq:nn-autonomous-format}.
\end{proof}

For fixed activation formats, the chain degree grows linearly with depth and the chain length grows linearly with the number of hidden units. The affine-output assertion is separate because a nonconstant affine function is not autonomous with respect to an empty chain.

\begin{example}[Activation functions]\label{ex:activations}
Table~\ref{tab:activations} lists common activation functions and their Pfaffian formats. Here $\sigma(x)=(1+\econst^{-x})^{-1}$ is the logistic sigmoid, $\phi$ and $\Phi$ denote the standard Gaussian density and cdf, respectively, and $\mathrm{sp}(x)=\ln(1+\econst^x)$ is the softplus function~\cite{glorot2011deep}. GELUs were introduced in~\cite{hendrycks2016gaussian}. Nonsmooth activation functions such as ReLU do not fall into our framework, nor does the ELU, which is not analytic at the origin.

\begin{table}[htbp]
  \centering
  \begin{tabular}{lllcccc}
    \toprule
    Activation & Function & Chain $\boldq$ & $\alpha$ & $\beta$ & $s$ & Auton.\\
    \midrule
    Exponential & $\econst^x$ & $(\econst^x)$ & 1 & 1 & 1 & Yes \\
    Sigmoid & $(1+\econst^{-x})^{-1}$ & $(\sigma)$ & 2 & 1 & 1 & Yes \\
    Tanh & $\tanh(x)$ & $(\tanh)$ & 2 & 1 & 1 & Yes \\
    Softplus & $\ln(1+\econst^x)$ & $(\sigma,\mathrm{sp})$ & 2 & 1 & 2 & Yes \\
    SiLU/Swish & $x\sigma(x)$ & $(\sigma)$ & 2 & 2 & 1 & No \\
    GELU & $x\Phi(x)$ & $(\phi,\Phi)$ & 2 & 2 & 2 & No \\
    Mish & $x\tanh(\mathrm{sp}(x))$ & $(\sigma,\mathrm{sp},\tanh(\mathrm{sp}))$ & 3 & 2 & 3 & No \\
    Gaussian & $\econst^{-x^2}$ & $(\econst^{-x^2})$ & 2 & 1 & 1 & No \\
    Arctan & $\arctan(x)$ & $((1+x^2)^{-1},\arctan)$ & 3 & 1 & 2 & No \\
    \bottomrule
  \end{tabular}
  \caption{Pfaffian formats of common activation functions. ``Auton.'' means autonomous with respect to the displayed chain.}
  \label{tab:activations}
\end{table}
\end{example}

\begin{example}[Softmax]\label{ex:softmax}
A common output function for classification problems is the softmax function. Each component of the softmax function
\begin{equation}\label{eq:softmax}
  g(x) = \left(\frac{\econst^{x_1}}{\sum_{j}\econst^{x_j}}, \dots, \frac{\econst^{x_m}}{\sum_{j}\econst^{x_j}}\right)
\end{equation}
is Pfaffian with format $(3,1,m)$, but the Pfaffian chains are not identical. For example, if we set
\[
q_{ij}(x) = \econst^{x_j - x_i}
\]
then we have
\begin{equation*}
  \frac{\partial g_i}{\partial x_j}(x) = \begin{cases}
    -q_{ij} (x) g_i(x)^2 & \text{ if } i\neq j\\
    g_i(x)(1-g_i(x)) & \text{ if } i=j.
  \end{cases}
\end{equation*}
In this case, a Pfaffian chain for $g_i$ is given by $(q_{i1}, \dotsc , \widehat{q_{ii}}, \dotsc, q_{im}, g_i)$. Note that this also gives a Pfaffian chain for any other $g_j$, but yielding a format $(3, 2, m)$ as one has to write $g_j = g_i q_{ij}$.
\end{example}

For classification, write $g_{ij}=F_j-F_i$ for score differences. When softmax is used, let $z_i$ denote the logits and write $\widetilde g_{ij}=z_j-z_i$ for logit differences. For affine outputs without softmax, take $z=F$, so $\widetilde g_{ij}=g_{ij}$.

Softmax preserves every ordering and equality of the logits. More precisely, on $\{x:z_i(x)=z_j(x)\}$,
\[
\nabla g_{ij}(x)=\frac{\econst^{z_i(x)}}{\sum_k\econst^{z_k(x)}}\,
\nabla\widetilde g_{ij}(x).
\]
Thus pairwise equality hypersurfaces and their regularity can be analysed using the logits, without increasing the Pfaffian format.

The closed score regions are
\[
C_j=\{x\colon g_{ij}(x)\ge0\text{ for every }i\ne j\}.
\]
An input with a unique maximal score is assigned that class; ties may be resolved by a fixed tie-breaking rule. Set
\[
\Sigma_j=\bigcup_{i\ne j}(C_j\cap C_i),
\qquad
\Sigma=\bigcup_j\Sigma_j.
\]
For arbitrary scores, $\Sigma$ need not coincide with the topological boundary of the regions. Under the assumption $\nabla g_{ij}\ne0$ on $\mathcal Z(g_{ij})$ for every pair, however, a tie in $C_j$ has arbitrarily close points where $g_{ij}<0$. Consequently $\inter C_j$ consists exactly of the unique maxima for class $j$, and $\partial C_j=\Sigma_j$.

\begin{proposition}\label{prop:decision-pfaffian}
If $F:\R^n\to\R^m$ is Pfaffian, the pairwise differences $g_{ij}$ are Pfaffian and the regions $C_j$ and tie loci $\Sigma_j,\Sigma$ are semi-Pfaffian.
\end{proposition}

\begin{proof}
This follows from Lemma~\ref{le:pfaffalgebra} and the defining finite Boolean combinations of Pfaffian equalities and inequalities.
\end{proof}

\begin{corollary}\label{cor:nn-pfaff-tube}
Suppose a network has $\ell$ hidden layers, $h$ hidden units, autonomous activations of format $(2,1,1)$, and affine outputs, optionally followed by softmax. Let $V$ be a pairwise logit equality hypersurface whose defining difference has $0$ as a regular value. For $X$ uniform in $B(p,\rho)$,
\[
\prob\{\dist(X,V)\le\varepsilon\}
\le C_{2\ell,1,h,n}
\left[\left(1+(2\ell+1)\frac{\varepsilon}{\rho}\right)^n
-\left(1+\frac{\varepsilon}{\rho}\right)^n\right].
\]
\end{corollary}

\begin{proof}
The pairwise logit differences share the format $(2\ell,1,h)$ by Proposition~\ref{prop:nn-format}. Apply Theorem~\ref{thm:prob_bound}.
\end{proof}

The constant still contains the factor $2^{h(h-1)/2}$. We next exploit the structure of a single sigmoid layer to replace the width dependence by a polynomial.

\subsection{The Gauss map of a neural network}
Single-hidden-layer sigmoid networks are universal approximators~\cite{cybenko1989approximation,hornik1989multilayer}. For target functions with finite first Fourier moment, Barron's theorem further gives an $L^2$-norm approximation error of order $w^{-1/2}$, with a target-dependent constant~\cite{barron1993universal}. These facts motivate studying width as the main complexity parameter.

The Bernstein--Kushnirenko--Khovanskii theorem~\cite{bernstein1975,cox2005using} bounds isolated solutions of Laurent polynomial systems by mixed volumes. A single exponential substitution turns a sigmoid network with rational weights into such a system.

\begin{proposition}\label{prop:single-layer-degree}
Let
\[
f(x)=c_0+\sum_{k=1}^w d_k\sigma(a_k^{\trans}x+b_k),
\qquad \sigma(t)=(1+\econst^{-t})^{-1},
\]
where $w\ge1$, $a_k\in\Q^n$, and $c_0,d_k,b_k\in\R$. Choose a common denominator $q\in\N_{>0}$, put $p_k=qa_k\in\Z^n$, and suppose $L=\max_k\|p_k\|_\infty\ge1$. If $0$ is a regular value of $f$, then $V=\mathcal Z(f)$ satisfies
\begin{equation}\label{eq:single-layer-bound}
\md(V)\le n!L^nw^n.
\end{equation}
\end{proposition}

\begin{proof}
Write $p_k=(p_{k1},\dots,p_{kn})$ and $y^a=\prod_{j=1}^n y_j^{a_j}$ for $a\in\Z^n$. Set $y_j(x)=\econst^{-x_j/q}$, $c_k=\econst^{-b_k}>0$, and
\begin{align}
D(y)&=\prod_{k=1}^w(1+c_ky^{p_k}),\nonumber\\
G(y)&=c_0D(y)+\sum_{k=1}^w d_k\prod_{l\ne k}(1+c_ly^{p_l}).
\label{eq:laurent-numerator}
\end{align}
Then $G(y(x))=D(y(x))f(x)$, and $D$ is strictly positive on the positive orthant. Fix any regular value $v$ of $\gamma_V$ and, after permuting coordinates, assume $v_n\ne0$. The Gauss fibre is described by
\begin{align}
G(y)&=0,\label{eq:laurent-system}\\
v_ny_j\partial_{y_j}G(y)-v_jy_n\partial_{y_n}G(y)&=0,
\qquad 1\le j<n.\label{eq:laurent-system-j}
\end{align}
Under this change of variables, $\partial_{x_j}=-(y_j/q)\partial_{y_j}$. Thus, when $f(x)=0$, the left-hand side of~\eqref{eq:laurent-system-j}, evaluated at $y=y(x)$, equals
\[
-qD(y(x))\bigl(v_n\partial_j f(x)-v_j\partial_n f(x)\bigr).
\]
Together with $G(y(x))=D(y(x))f(x)$, this is an invertible triangular change of equations at a zero: the additional terms are multiples of $f$, and the diagonal factors are $D(y(x))$ and $-qD(y(x))$. The exponential substitution is a diffeomorphism onto $\R_{>0}^n$. Lemma~\ref{le:gauss-nondegenerate} therefore shows that every real solution counted here is non-degenerate, and hence is also an isolated complex solution of the Laurent system.

All supports in~\eqref{eq:laurent-system}--\eqref{eq:laurent-system-j} lie in the zonotope (a Minkowski sum of line segments)
\[
Z=\sum_{k=1}^w[0,p_k].
\]
If one of these equations is identically zero, non-degeneracy implies that the real fibre is empty. Otherwise, Bernstein's theorem bounds the number of isolated torus solutions by
$n!\operatorname{MV}(P_0,\dots,P_{n-1})$, where $P_j$ is the Newton polytope of the $j$-th equation. We use Euclidean mixed volume, with $\operatorname{MV}(P,\dots,P)=\vol(P)$. Its monotonicity gives an upper bound $n!\vol(Z)$. The width of $Z$ in coordinate $j$ is $\sum_k|p_{kj}|\le Lw$, so
\begin{equation}\label{eq:zonotope-volume}
\vol(Z)\le(Lw)^n.
\end{equation}
This bounds every regular fibre and proves~\eqref{eq:single-layer-bound}.
\end{proof}

\begin{remark}
The parameter $q$ clears the denominators of the first-layer weights, while $L$ bounds their resulting integer coordinates. All statements about polynomial dependence on width hold for fixed $n,L$. Biases and output coefficients remain arbitrary real numbers.
\end{remark}

We next show that the exponent $n$ in~\eqref{eq:single-layer-bound} cannot be lowered.

\begin{lemma}\label{lem:bounded-gauss-surjective}
If $M$ is a smooth compact hypersurface bounding a nonempty bounded open set $W$, then $|\gamma_M^{-1}(v)|\ge2$ for every $v\in S^{n-1}$.
\end{lemma}

\begin{proof}
The height $x\mapsto\langle x,v\rangle$ attains distinct maximum and minimum values on $\overline W$, at points $x_+,x_-\in\partial W=M$. Its derivative along $M$ vanishes there, so $v$ is normal at both points. The generalised Gauss fibre contains the two distinct points $(x_+,v)$ and $(x_-,v)$.
\end{proof}

\begin{proposition}\label{prop:single-layer-lowerbound}
For every $n\ge1$ and every even $r\ge4$, there is a single-hidden-layer sigmoid network with $w=n(r+2)$ units and lattice constant $L=1$ for which $0$ is a regular value and $V=\mathcal Z(f)$ is compact, with
\begin{equation}\label{eq:single-layer-lb}
\md(V)\ge2\left(\frac{r-2}{2}\right)^n=\Omega_n(w^n).
\end{equation}
\end{proposition}

\begin{proof}
For $\lambda>0$, define
\begin{align*}
h_\lambda(t)&=\sum_{k=1}^r(-1)^k\sigma(\lambda(t-k)),\\
\hat h_\lambda(t)&=h_\lambda(t)-2\sigma(\lambda(t-(r+1)))-2\sigma(-\lambda t),\\
H_\lambda(x)&=\sum_{j=1}^n\hat h_\lambda(x_j),\qquad
f_\lambda(x)=\tfrac12+\delta+H_\lambda(x),
\end{align*}
where $\lambda$ will be chosen first and then $\delta\in(0,1/8)$. This uses $w=n(r+2)$ units. Pairing consecutive terms gives
\[
h_\lambda(t)=\sum_{k=1}^{r/2}
\bigl[\sigma(\lambda(t-2k))-\sigma(\lambda(t-2k+1))\bigr]<0.
\]
Thus $\hat h_\lambda<0$, and $\hat h_\lambda(t)\to-2$ as $t\to\pm\infty$. If one coordinate of $x$ has sufficiently large absolute value, its contribution is less than $-1$ and all others are negative, so $f_\lambda(x)<0$. Consequently $\mathcal Z(f_\lambda)$ is compact.

For every multi-index $a=(a_1,\dots,a_n)\in\{2,4,\dots,r-2\}^n$, consider the core and enlarged box
\[
K_a=\prod_{j=1}^n[a_j+\tfrac14,a_j+\tfrac34],
\qquad
R_a=\prod_{j=1}^n[a_j-\tfrac12,a_j+\tfrac32].
\]
On the finitely many cores, $H_\lambda\to0$ uniformly as $\lambda\to\infty$. On each face of $R_a$, at least one coordinate is a fixed midpoint of an odd interval, where $\hat h_\lambda\to-1$. All other coordinate contributions remain strictly negative, even at the sigmoid thresholds. Choose $\lambda$ large enough that $H_\lambda>-1/4$ on every core and these face-coordinate values are all below $-7/8$. Then, simultaneously for all $\delta\in(0,1/8)$,
\[
f_\lambda>1/4\text{ on }K_a,
\qquad f_\lambda<-1/4\text{ on }\partial R_a.
\]

With this $\lambda$ fixed, choose $\delta\in(0,1/8)$ so that $-1/2-\delta$ is a regular value of $H_\lambda$, which is possible by Sard's theorem. Each core belongs to a connected component $W_a$ of $\{f_\lambda>0\}$ contained in $\inter R_a$. The interiors of these boxes are pairwise disjoint. Since $0$ is regular, near every zero the positive side is connected; hence boundaries of distinct $W_a$ cannot meet. Each $\partial W_a$ is a smooth compact hypersurface, possibly disconnected, contained in $\mathcal Z(f_\lambda)$. There are $N=((r-2)/2)^n$ such regions. Lemma~\ref{lem:bounded-gauss-surjective} gives at least $2N$ distinct points in every regular Gauss fibre of $\mathcal Z(f_\lambda)$.

\begin{figure}[htbp]
  \centering
  \includegraphics[width=0.4\textwidth]{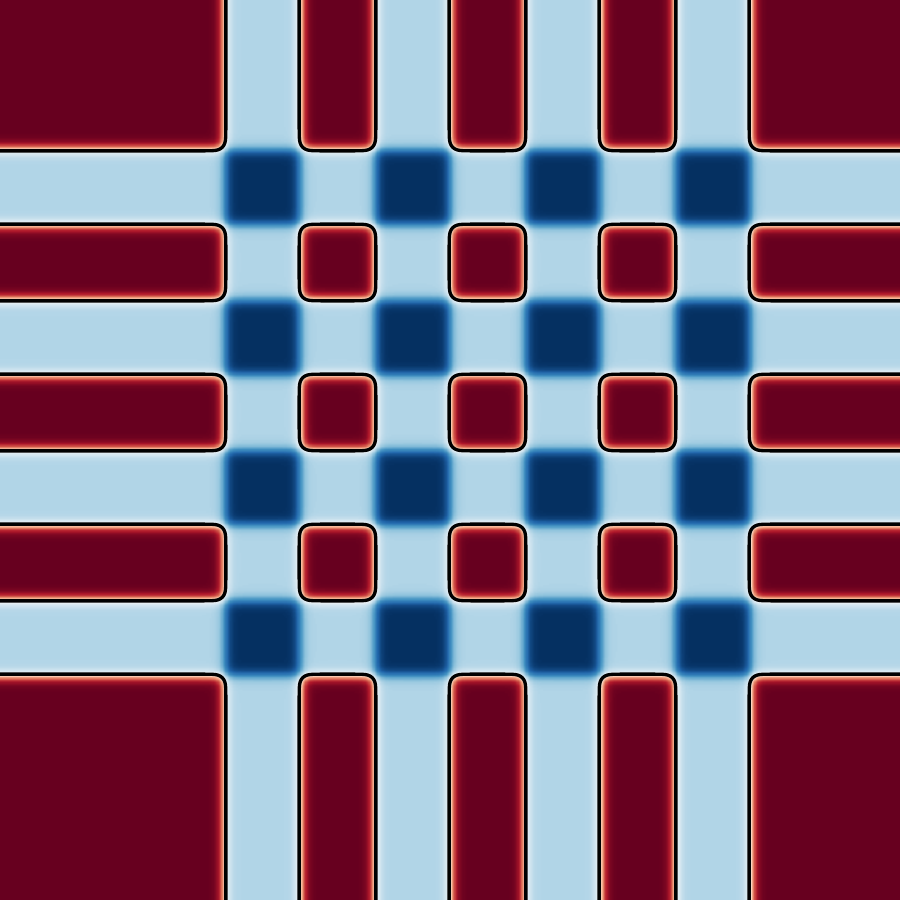}
  \caption{Visualization of the lower-bound construction in Proposition~\ref{prop:single-layer-lowerbound} for $n=2$, shown without the capping units: the function plotted is $\tfrac12+\delta+h_\lambda(x_1)+h_\lambda(x_2)$. The black lines indicate its zero set and the red regions are the regions where it is positive. The unbounded components visible towards the boundary of the picture are removed by the capping units.}
  \label{fig:lower-bound-construction}
\end{figure}

Finally, let $f(x)=f_\lambda(x/\lambda)$. Explicitly,
\[
f(x)=\tfrac12+\delta+\sum_{j=1}^n
\left[\sum_{k=1}^r(-1)^k\sigma(x_j-\lambda k)
-2\sigma(x_j-\lambda(r+1))-2\sigma(-x_j)\right].
\]
Its weights belong to $\{\pm e_1,\dots,\pm e_n\}$, so $q=L=1$. Its zero set is the image $\lambda\mathcal Z(f_\lambda)$, which preserves normal directions and Gauss-fibre cardinalities. Regularity and compactness are also preserved, proving~\eqref{eq:single-layer-lb}.
\end{proof}

\begin{remark}\label{rem:why-not-fewnomials}
The Bernstein bound uses the volume of the common support polytope. Real fewnomial estimates such as~\cite{bihansottile2007} instead depend on the number of monomials. These are different measures of complexity, so replacing Bernstein's theorem by a fewnomial bound does not automatically improve~\eqref{eq:single-layer-bound}.
\end{remark}

For greater depth, sigmoid composition prevents the direct Laurent-polynomial reduction. We formulate a fixed-depth extension as a conjecture.

\begin{conjecture}\label{conj:multi-layer-degree}
Fix $n,\ell,L\ge1$. Let $F:\R^n\to\R^m$ be a network with $\ell$ hidden layers of widths $n_1,\dots,n_\ell\ge n$, logistic sigmoid activations and affine outputs. Suppose the hidden-layer weight matrices have a common denominator $q\in\N_{>0}$, with all scaled entries bounded in absolute value by $L$. Then there is a constant $C(n,\ell,L)$ such that each pairwise difference with $0$ as a regular value satisfies
\begin{equation}\label{eq:multi-layer-bound}
\md\bigl(\mathcal Z(g_{ij})\bigr)
\le C(n,\ell,L)\left(\prod_{k=1}^\ell n_k\right)^n.
\end{equation}
Biases and affine-output coefficients may be arbitrary real numbers.
\end{conjecture}

Proposition~\ref{prop:single-layer-degree} proves the case $\ell=1$, with optimal width exponent by Proposition~\ref{prop:single-layer-lowerbound}. As a comparison, for fixed input dimension and depth, the number of linear regions of ReLU networks is bounded by $O_{n,\ell}(\prod_k n_k^n)$~\cite{raghu2017,serra2018,hanin2019}. ReLU networks lie outside the smooth Pfaffian setting; this comparison motivates the form of the conjecture but does not prove it.

The exponential growth in decision-region topology with depth established by Bianchini and Scarselli~\cite[Theorem~5 and Appendix~F]{bianchini2014complexity} is compatible with this conjecture, which concerns width dependence at fixed depth.

\subsection{Tubular neighbourhood bound for neural networks}\label{subsec:tube-nn}
A tube bound needs all section degrees $\md_i(V)$, not just $\md(V)$. An arbitrary affine section changes rational weight vectors into vectors with possibly irrational entries, so the preceding Laurent-polynomial argument does not directly give a uniform bound for sections. We instead retain the original coordinate exponentials as a global Pfaffian chain of length $n$.

\begin{proposition}\label{prop:single-layer-ideg}
Under the hypotheses of Proposition~\ref{prop:single-layer-degree}, put $D_0=2nLw$. For every $0\le i<n$,
\begin{equation}\label{eq:single-layer-ideg-precise}
\md_i(V)\le
2^{\frac{n(n-1)}2}D_0^{i+1}\bigl((i+1)D_0+1\bigr)^n.
\end{equation}
In particular,
\begin{equation}\label{eq:single-layer-ideg}
\md_i(V)\le K(n,L)w^{2n},\qquad
K(n,L):=2^{\frac{n(n-1)}2}(2nL)^n(2n^2L+1)^n.
\end{equation}
\end{proposition}

\begin{proof}
Use $G,D,y$ from~\eqref{eq:laurent-numerator}. Every exponent of $G$ lies in $[-Lw,Lw]^n$, so
\[
P(y)=(y_1\cdots y_n)^{Lw}G(y)
\]
is an ordinary polynomial of total degree at most $D_0$. The function
\[
\widetilde f(x)=P(y(x))=a(x)f(x),
\qquad
a(x)=(y_1(x)\cdots y_n(x))^{Lw}D(y(x))>0
\]
has the same regular zero set as $f$.

Parametrise a transverse affine section by $x=x_0+Ut$, where $t\in\R^d$, $d=i+1$, and $U\in\R^{n\times d}$ satisfies $U^{\trans}U=I_d$. The functions
\[
\eta_j(t)=\exp\!\left(-\frac{x_{0j}+(Ut)_j}{q}\right),\qquad 1\le j\le n,
\]
form a global autonomous chain on $\R^d$ of length $n$ and chain degree $1$, since
$\partial_{t_k}\eta_j=-(U_{jk}/q)\eta_j$. Put $\widehat f(t)=P(\eta_1(t),\dots,\eta_n(t))$. Both $\widehat f$ and its directional derivatives $\partial_\tau\widehat f$, for constant $\tau\in\R^d$, have degree at most $D_0$ in this chain. At each regular Gauss value, Lemma~\ref{le:gauss-nondegenerate} gives a non-degenerate system of $d$ such equations. Applying Theorem~\ref{thm:khovanskii} in $d$ variables gives
\[
2^{\frac{n(n-1)}2}D_0^d(dD_0-d+\min\{d,n\}+1)^n
=2^{\frac{n(n-1)}2}D_0^d(dD_0+1)^n.
\]
Taking suprema proves~\eqref{eq:single-layer-ideg-precise}. Since $1\le d\le n$ and $w,L\ge1$, this is at most the right-hand side of~\eqref{eq:single-layer-ideg}.
\end{proof}

The chain has length $n$ regardless of the width, and remains global even for irrational section directions. No logarithmic coordinates or constraints on the domain are needed.

\begin{theorem}\label{thm:prob_bound_single_layer}
Let $f,V$ satisfy Proposition~\ref{prop:single-layer-degree}, and let $X$ be uniform in any ball $B(p,\rho)$ with $\rho>0$. For every $\varepsilon>0$, writing $u=\varepsilon/\rho$,
\begin{equation}\label{eq:single-layer-prob-bound}
\prob\{\dist(X,V)\le\varepsilon\}
\le 2K(n,L)w^{2n}\bigl[(1+2u)^n-(1+u)^n\bigr].
\end{equation}
If in addition $V\subset B(p,\rho)$, the bracket can be replaced by $(1+u)^n-1$.
\end{theorem}

\begin{proof}
Substitute~\eqref{eq:single-layer-ideg} into Proposition~\ref{prop:local-tube} and use
\[
\sum_{i=0}^{n-1}\binom n{i+1}u^{i+1}(1+u)^{n-i-1}
=(1+2u)^n-(1+u)^n.
\]
If $V\subset B(p,\rho)$, apply Theorem~\ref{thm:tube_bound} directly to $V$ and divide by $\omega_n\rho^n$.
\end{proof}

A one-variable zero count improves the leading coefficient.

\begin{lemma}\label{lem:polya}
A nonzero exponential polynomial $E(t)=\sum_{j=1}^M d_j\econst^{\nu_jt}$ with distinct real frequencies has at most $M-1$ distinct real zeros.
\end{lemma}

\begin{proof}
Discard zero coefficients and argue by induction on $M$. Multiplication by $\econst^{-\nu_1t}$ preserves zeros, and differentiation leaves an exponential polynomial with $M-1$ frequencies. Rolle's theorem bounds the zeros of the original function by one more than those of its derivative.
\end{proof}

\begin{lemma}\label{lem:line-intersection}
For $f$ as in Proposition~\ref{prop:single-layer-degree}, its restriction to any affine line has at most $(2Lw+1)^n-1$ isolated real zeros.
\end{lemma}

\begin{proof}
On a line $x=x_0+tv$, multiply $f(x)$ by the positive denominator $D(y(x))$. Expanding~\eqref{eq:laurent-numerator} gives an exponential polynomial whose frequencies are
\[
-\frac1q\left\langle\sum_{k=1}^w \kappa_kp_k,v\right\rangle,
\qquad \kappa\in\{0,1\}^w.
\]
The integer vectors $\sum_k\kappa_kp_k$ lie in $[-Lw,Lw]^n$, so there are at most $(2Lw+1)^n$ distinct frequencies. If the resulting polynomial is nonzero, apply Lemma~\ref{lem:polya}; if it is identically zero, the restriction has no isolated zeros.
\end{proof}

\begin{corollary}\label{cor:single-layer-tight-prob-bound}
Under the hypotheses of Theorem~\ref{thm:prob_bound_single_layer}, let $a_w=(2Lw+1)^n-1$. Then
\begin{align}
\prob\{\dist(X,V)\le\varepsilon\}\le{}&
2na_wu(1+u)^{n-1}\nonumber\\
&+2K(n,L)w^{2n}
\bigl[(1+2u)^n-(1+u)^n-nu(1+u)^{n-1}\bigr].
\label{eq:single-layer-refined-tube}
\end{align}
In particular, for all $u>0$,
\begin{equation}\label{eq:single-layer-linear-tube}
\prob\{\dist(X,V)\le\varepsilon\}
\le\min\{1,H(n,L)w^nu\},
\end{equation}
where one may take
\begin{equation}\label{eq:single-layer-H}
H(n,L)=2^nn(2L+1)^n+
2K(n,L)\bigl(3^n-2^n-n2^{n-1}\bigr).
\end{equation}
\end{corollary}

\begin{proof}
Transverse line sections consist of isolated zeros, so Lemma~\ref{lem:line-intersection} gives $\md_0(V)\le a_w$. Use this for the $i=0$ term of~\eqref{eq:local-tube}, and~\eqref{eq:single-layer-ideg} for the remaining terms, to obtain~\eqref{eq:single-layer-refined-tube}.

Put $z=w^nu$. If $z\le1$, then $u\le1$ and the first term is at most $2^nn(2L+1)^nz$. For the remaining terms, $u^{i+1}\le u^2$ and $(1+u)^{n-i-1}\le2^{n-i-1}$ for $i\ge1$. Thus their sum is at most
\[
2K(n,L)z^2
\sum_{i=1}^{n-1}\binom n{i+1}2^{n-i-1}
=2K(n,L)z^2(3^n-2^n-n2^{n-1}).
\]
Use $z^2\le z$. If $z>1$, the trivial probability bound suffices since $H(n,L)\ge1$. This proves~\eqref{eq:single-layer-linear-tube}.
\end{proof}

For $u\to0$, the refined bound has expansion $2na_wu+O_{n,L}(w^{2n}u^2)$, with the remainder uniform for $w\ge1$ and $0\le u\le1$. The width exponent in this probability bound is an upper estimate; Proposition~\ref{prop:single-layer-lowerbound} establishes sharpness only for the Gauss-map degree.

The regular-value hypothesis can be removed from the linear-radius estimate by approximating the zero set from both signs using an elementary hypersurface approximation principle; compare the discussion of nearby levels in the introduction of~\cite{basu2021hausdorff}. Using two levels of $f$ preserves the shallow network architecture.

\begin{corollary}\label{cor:single-layer-singular-linear}
Let $f$ have the representation and weight assumptions of Proposition~\ref{prop:single-layer-degree}, but replace its regular-value hypothesis by $f\neq0$. Put $V=\mathcal Z(f)$, and let $X$ be uniform in $B(p,\rho)$, where $\rho>0$. Then, for every $\varepsilon>0$,
\[
\prob\{\dist(X,V)\le\varepsilon\}
\le\min\left\{1,2H(n,L)w^n\frac{\varepsilon}{\rho}\right\},
\]
where $H(n,L)$ is given by~\eqref{eq:single-layer-H}.
\end{corollary}

\begin{proof}
If $V$ is empty, the assertion is immediate. Otherwise $f$ is nonconstant. By Sard's theorem, choose $\delta_j\rightarrow 0$ such that both $\delta_j$ and $-\delta_j$ are regular values of $f$, and set
\[
V_j^+=\{f=\delta_j\},\qquad
V_j^-=\{f=-\delta_j\},\qquad A_j=V_j^+\cup V_j^-.
\]
Each nonempty $V_j^\pm$ is a regular zero set of a shallow network with the same $n,w,L$ as $f$: only the affine-output bias changes. Consequently, Corollary~\ref{cor:single-layer-tight-prob-bound} and the union bound give, for every $r>0$,
\begin{equation}\label{eq:two-level-tube}
\prob\{\dist(X,A_j)\le r\}
\le\min\{1,2H(n,L)w^n r/\rho\}.
\end{equation}

For every $z\in V$, we have $\dist(z,A_j)\to0$. Indeed, analyticity and $f\not\equiv0$ imply that every neighbourhood of $z$ contains a point $y$ with $f(y)\ne0$. For all sufficiently large $j$, the intermediate value theorem along the segment from $z$ to $y$ gives a point of $A_j$ on that segment. Taking $y$ arbitrarily close to $z$ proves the claim. Both signs are needed because $f$ may have a local maximum or minimum at a zero.

Since $V$ is nonempty and closed, a nearest point in $V$ exists for every $x\in\R^n$. The preceding claim therefore implies
\[
\limsup_{j\to\infty}\dist(x,A_j)\le\dist(x,V).
\]
For every $\eta>0$, it follows that
\[
\mathbf1_{\{\dist(x,V)\le\varepsilon\}}
\le\liminf_{j\to\infty}
\mathbf1_{\{\dist(x,A_j)\le\varepsilon+\eta\}}.
\]
Apply Fatou's lemma under the uniform probability measure, use~\eqref{eq:two-level-tube} with $r=\varepsilon+\eta$, and let $\eta$ go to zero.
\end{proof}

\begin{remark}
This estimates the full distance neighbourhood, including near singular points. If $V$ has real codimension $c>1$, it still gives only linear decay in $\varepsilon/\rho$; it does not give the sharper $(\varepsilon/\rho)^c$ behaviour sought in dimension-sensitive tube bounds. The exclusion $f\equiv0$ is essential, since then $V=\R^n$ and the probability is one at every radius.
\end{remark}

\section{Robustness}\label{sec:robustness}
Let $F:\R^n\to\R^m$ be a classifier with $m\ge2$ scores and top-score tie locus
\[
\Sigma=\bigcup_{i<j}(C_i\cap C_j).
\]
Writing $g_{ij}=F_j-F_i$ and $V_{ij}=\mathcal Z(g_{ij})$, we have
\begin{equation}\label{eq:sigma-inclusion}
\Sigma\subseteq\bigcup_{i<j}V_{ij}.
\end{equation}
The inclusion may be strict because $V_{ij}$ also records ties between nonmaximal scores. Our estimates control proximity to this larger union.

Under the regularity assumptions below, each $V_{ij}$ is closed and has no boundary, so $T(V_{ij},\varepsilon)=T^\perp(V_{ij},\varepsilon)$. The normal-tube estimates therefore control the full distance neighbourhoods needed here, even though $\Sigma$ itself may be nonsmooth.

For $x\notin\Sigma$, put
\[
\Delta(x)=\dist(x,\Sigma).
\]
This is the distance to loss of a unique maximal score. Under the pairwise regularity assumptions below, it is also the infimum of perturbation sizes that change the predicted label. Indeed, starting from $x\in\inter C_j$, any path to another label meets $\partial C_j\subseteq\Sigma$, while arbitrarily close to each point of $\partial C_j=\Sigma_j$ there are points outside $C_j$. A segment from $x$ to any point of $\Sigma$ either ends on $\partial C_j$ or meets it earlier. Thus the distance to $\Sigma$ equals the distance to $\partial C_j$, and hence the infimum of distances to a different label. This concerns label stability, not agreement with ground truth.

\begin{definition}\label{def:condition}
The \emph{condition number} of the classifier at $x\notin\Sigma$ is
\[
\mC(x)=\mC_0(x)=\frac{\|x\|}{\Delta(x)}.
\]
More generally, its \emph{local condition number} centred at $p\in\R^n$ is
\[
\mC_p(x)=\frac{\|x-p\|}{\Delta(x)}.
\]
\end{definition}

We use $\dist(x,\varnothing)=\infty$, with the corresponding condition number equal to zero. Under pairwise regularity, $\Sigma$ has Lebesgue measure zero. The condition number may be extended by $\infty$ on $\Sigma\setminus\{p\}$ and defined arbitrarily at $p$; these choices do not affect the probabilities considered below.

For $x\ne p$ off $\Sigma$, the inequality $\mC_p(x)>t$ means that the unique predicted label can be lost under a perturbation of relative size less than $1/t$, with size measured relative to $\|x-p\|$. This parallels the interpretation of distance-based condition numbers in numerical analysis~\cite{burgisser2013condition}.

\begin{remark}[Weak condition number]\label{rem:weak-condition}
For a uniformly random direction $v\in S^{n-1}$, define the first distance to a tie by
\[
\Delta(x;v)=\inf\{\eta>0:x+\eta v\in\Sigma\},
\]
with value $\infty$ if the ray does not meet $\Sigma$. For $0<\delta<1$, the quantile
\[
\mC^{\mathrm{weak}}_\delta(x)=\inf\left\{\eta\ge0:
\prob_v\left\{\frac{\|x\|}{\Delta(x;v)}>\eta\right\}\le\delta\right\}
\]
gives a directional analogue of the weak condition numbers in~\cite{lotz2020weak}. 
\end{remark}

\subsection{Tail bounds on the condition number}
The estimates below apply both to random inputs and to random perturbations of a fixed centre. The local normalisation makes the bounds independent of the sampling radius.

\begin{theorem}[Uniform tail bound]\label{thm:condition-uniform}
Suppose every pairwise difference $g_{ij}:\R^n\to\R$ is Pfaffian of format $(\alpha,\beta,s)$, with $\alpha,\beta\ge1$, and has $0$ as a regular value. If $X$ is uniform in $B(p,\rho)$, then for every $t>0$,
\begin{equation}\label{eq:general-condition-tail}
\prob\{\mC_p(X)>t\}\le
\min\left\{1,\binom m2 C_{\alpha,\beta,s,n}
\left[\left(1+\frac{\alpha+\beta}{t}\right)^n
-\left(1+\frac1t\right)^n\right]\right\},
\end{equation}
where $C_{\alpha,\beta,s,n}$ is defined in~\eqref{eq:pfaffian-tube-constant}.
\end{theorem}

\begin{proof}
On $B(p,\rho)$, the event $\mC_p(X)>t$ implies $\Delta(X)<\rho/t$. By~\eqref{eq:sigma-inclusion}, it is contained, up to a null set, in
\[
\bigcup_{i<j}\{\dist(X,V_{ij})<\rho/t\}.
\]
Apply Theorem~\ref{thm:prob_bound} to each pair with $\varepsilon=\rho/t$, sum the bounds, and use that a probability is at most one.
\end{proof}

A Gaussian distribution is an exact mixture of uniform-ball distributions. We state this independently of the set to which it will be applied.

\begin{lemma}\label{lem:gauss-to-uniform}
Let $X\sim\mathcal N(p,\varsigma^2 I_n)$, with $\varsigma>0$, and let $U_r$ be uniform in $B(p,r)$. For every measurable $A\subseteq\R^n$,
\begin{equation}\label{eq:gauss-to-uniform}
\prob\{X\in A\}
=\frac{\omega_n}{(2\pi)^{n/2}}
\int_0^\infty\prob\{U_{\varsigma r}\in A\}
\,r^{n+1}\econst^{-r^2/2}\,\diff r.
\end{equation}
In particular, any upper bound for $\prob\{U_r\in A\}$ valid for every $r>0$ also bounds $\prob\{X\in A\}$.
\end{lemma}

\begin{proof}
Substitute
\[
\prob\{U_{\varsigma r}\in A\}
=\frac{1}{\omega_n(\varsigma r)^n}
\int_A\mathbf1_{\{\|x-p\|\le\varsigma r\}}\,\diff x
\]
into the right-hand side and apply Tonelli's theorem. The inner integral is
\[
\int_{\|x-p\|/\varsigma}^\infty r\econst^{-r^2/2}\,\diff r
=\exp\!\left(-\frac{\|x-p\|^2}{2\varsigma^2}\right),
\]
giving the Gaussian density on $A$. Taking $A=\R^n$ proves the normalisation.
\end{proof}

\begin{theorem}[Gaussian tail bound]\label{thm:condition-gaussian}
Under the pairwise format and regularity assumptions of Theorem~\ref{thm:condition-uniform}, the bound~\eqref{eq:general-condition-tail} also holds for $X\sim\mathcal N(p,\varsigma^2 I_n)$, for every centre $p$ and every $\varsigma>0$.
\end{theorem}

\begin{proof}
Apply Lemma~\ref{lem:gauss-to-uniform} to the fixed set $A=\{x:\mC_p(x)>t\}$. Theorem~\ref{thm:condition-uniform} bounds its probability in every ball centred at $p$ by the same right-hand side.
\end{proof}

For fixed formats and $t\to\infty$, the expression inside the minimum in~\eqref{eq:general-condition-tail} has expansion
\[
\binom m2 C_{\alpha,\beta,s,n}\,
\frac{n(\alpha+\beta-1)}{t}+O_{\alpha,\beta,s,n,m}(t^{-2}).
\]

\subsection{Neural network classifiers}
Proposition~\ref{prop:nn-format} makes the general estimates explicit for neural networks.

\begin{corollary}\label{cor:nn-condition}
Let $F:\R^n\to\R^m$ have $\ell$ hidden layers and $h$ hidden units, with autonomous activation format $(\alpha,\beta,s)$, $\alpha,\beta,s\ge1$. Assume every pairwise score difference has $0$ as a regular value. For either $X$ uniform in $B(p,\rho)$ or $X\sim\mathcal N(p,\varsigma^2 I_n)$, the bound~\eqref{eq:general-condition-tail} holds with format $(\alpha_F,\beta_F,s_F)$ given by
\[
\begin{cases}
\bigl(\ell(\alpha+\beta-1)-\beta+1,\ \beta,\ sh\bigr),
&\text{for affine outputs},\\
\bigl(\ell(\alpha+\beta-1)+\alpha',\ \beta',\ sh+s'm\bigr),
&\text{for autonomous outputs},
\end{cases}
\]
where in the second case each output component has format $(\alpha',\beta',s')$ with positive parameters.
\end{corollary}

\begin{proof}
All pairwise differences share the common chain in Proposition~\ref{prop:nn-format}. Apply Theorems~\ref{thm:condition-uniform} and~\ref{thm:condition-gaussian}. If softmax follows affine outputs, apply the bounds to the logit differences $\widetilde g_{ij}$: their zero sets and regularity agree with those of the score differences $g_{ij}$, as established in Section~\ref{sec:tubular-neural}.
\end{proof}

\begin{example}[Sigmoid or tanh activation]\label{ex:sigmoid-condition}
For sigmoid or tanh hidden activations and affine outputs, $(\alpha_F,\beta_F,s_F)=(2\ell,1,h)$. Put
\[
B_{\ell,h,n}=1+(n-1)(2\ell-1)+2\ell\min\{n,h\}.
\]
Under pairwise regularity, for either sampling model,
\begin{align*}
\prob\{\mC_p(X)>t\}
&\le\min\left\{1,\binom m2
\frac{2^{h(h-1)/2}}{\ell}B_{\ell,h,n}^{\,h}
\left[\left(1+\frac{2\ell+1}{t}\right)^n
-\left(1+\frac1t\right)^n\right]\right\}.
\end{align*}
Before clipping, its leading term is
$\binom m2\,2^{h(h-1)/2+1}nB_{\ell,h,n}^{\,h}/t$.
\end{example}

For a single sigmoid layer, the refined tube bound gives a polynomial dependence on width for both sampling models, without a containment assumption.

\begin{corollary}\label{cor:nn-condition-single-layer}
Let $F:\R^n\to\R^m$ have one hidden layer of width $w\ge1$, logistic sigmoid activation and affine outputs, optionally followed by softmax. Suppose its first-layer weights $a_k$ have a common denominator $q\in\N_{>0}$, with $L=\max_k\|qa_k\|_\infty\ge1$, and every pairwise logit difference has $0$ as a regular value. For every $p\in\R^n$ and every $t>0$,
\begin{equation}\label{eq:single-layer-condition-bound}
\prob\{\mC_p(X)>t\}
\le\min\left\{1,\binom m2 H(n,L)\frac{w^n}{t}\right\},
\end{equation}
both for $X$ uniform in $B(p,\rho)$ with any $\rho>0$ and for $X\sim\mathcal N(p,\varsigma^2 I_n)$ with any $\varsigma>0$. Here $H(n,L)$ is the explicit constant in~\eqref{eq:single-layer-H}.
\end{corollary}

\begin{proof}
Every pairwise logit difference is of the form
\[
\widetilde g_{ij}(x)=c_{ij}+\sum_{k=1}^w d_{ij,k}\sigma(a_k^{\trans}x+b_k)
\]
with the same first-layer weights and lattice constant, and $V_{ij}=\mathcal Z(g_{ij})=\mathcal Z(\widetilde g_{ij})$. For uniform sampling, apply~\eqref{eq:single-layer-linear-tube} with $\varepsilon=\rho/t$ to each pair, and use the union bound as in Theorem~\ref{thm:condition-uniform}. The result is independent of $\rho$. Lemma~\ref{lem:gauss-to-uniform} therefore gives exactly the same estimate for Gaussian sampling.
\end{proof}

\begin{remark}\label{rem:single-layer-condition-tight}
A more detailed tail profile follows by replacing $u$ with $1/t$ in~\eqref{eq:single-layer-refined-tube} and multiplying by $\binom m2$, for either sampling model. In particular, for $t\ge1$ the resulting upper bound is
\[
\binom m2\left[
\frac{2n((2Lw+1)^n-1)}{t}
+O_{n,L}\!\left(\frac{w^{2n}}{t^2}\right)\right].
\]
The clipped estimate~\eqref{eq:single-layer-condition-bound} is valid for every $t>0$; its nontrivial regime starts at a constant multiple of $w^n$ when $n,L,m$ are fixed.
\end{remark}

\section{Conclusions}\label{sec:conclusions}
We have obtained local tube-volume bounds for regular Pfaffian hypersurfaces, including unbounded ones, and used them to bound a geometric condition number for classifiers under uniform and Gaussian sampling. For one hidden sigmoid layer with rational first-layer weights and fixed lattice constant, the tail is bounded by $O_{n,L,m}(w^n/t)$. The underlying Gauss-map degree bound $\md(V)\le n!L^nw^n$ has optimal width exponent, although sharpness of the probability estimate remains open.

\subsection{Improved bounds for neural networks}\label{subsec:improved-nn}
For general sigmoid networks, the Pfaffian bound still contains $2^{h(h-1)/2}$, with $h$ the number of hidden units. Conjecture~\ref{conj:multi-layer-degree} asks for polynomial bounds in the widths at fixed input dimension, depth and lattice constant. A corresponding estimate for every section degree would also improve the multilayer tube and robustness bounds.

The sigmoid chain has structure beyond a general Pfaffian chain. A chain element in layer $i$ has derivative polynomials of degree at most $2i$, depending on that element and on earlier layers, but not on other elements of its own layer. Moreover, for an input $u\in\R^{n_{i-1}}$ and preactivation $z^i=A^iu+b^i$, the differential of the layer map is
\[
\diff{F^i}(u)=\diag(\sigma'(z^i_1),\dots,\sigma'(z^i_{n_i}))A^i,
\qquad \sigma'(z^i_j)>0.
\]
It is natural to ask whether these properties permit a more efficient Rolle-type zero count. The usual Khovanskii argument (see, for example, the detailed proof in~\cite{lotz2026khovanskii}) processes chain elements successively; controlling the degree growth when processing a whole layer remains a difficulty. We do not assert a uniform improvement for arbitrary monotone Pfaffian activations: any such statement must account for the complexity of the activation itself, including its chain length.

Even for one hidden layer, sharpening the section bounds between $\md_0(V)$ and $\md_{n-1}(V)$ could reduce the constants and the higher-order coefficients in the tube estimate. The lower-bound construction for the top Gauss degree does not settle these questions.

\subsection{Non-smooth Pfaffian sets}\label{subsec:nonsmooth}
The main theorem requires $0$ to be a regular value of the defining function. Although the top-score tie locus can be nonsmooth where regular pairwise equality hypersurfaces meet, the union bound still applies in that situation. For a nonzero shallow sigmoid output, Corollary~\ref{cor:single-layer-singular-linear} removes the regular-value assumption from the linear-radius tube bound by passage to nearby regular levels. This retains the $O_{n,L}(w^n)$ coefficient, but does not recover the tube-radius power corresponding to a higher real codimension.

In the algebraic setting, Basu and Lerario~\cite{basu2021hausdorff} extend tube bounds to singular varieties through controlled smooth approximations and Hausdorff limits. Their construction uses genericity properties of complex algebraic polar varieties; see their Remark 1.2 for the difficulty of extending this method to the Pfaffian setting.

Pfaffian chains are real analytic and admit local holomorphic extensions satisfying the same differential identities. Moreover, critical loci of projections of regular Pfaffian sets can be expressed using Pfaffian Jacobian minors, with effective degree bounds. These local facts do not supply the global complex-algebraic genericity and deformation properties used in the algebraic approximation argument. Obtaining suitable replacements with controlled Pfaffian formats is an open direction for tube bounds that retain the real codimension. Covering and slicing methods, such as those of~\cite{zhang2025covering}, offer another route when one is willing to use different dimension-dependent constants.


\section*{AI Usage Statement}
The ideas and results in this paper predate the availability of generative AI for Mathematics. 
Anthropic's Claude Code and OpenAI's Codex were used in the final stages of this project to proofread and revise the document, as well as to simplify some of the arguments.

\raggedbottom
\bibliographystyle{alpha}
\bibliography{refs}

@misc{lotz2026khovanskii,
      title={Khovanskii's Bezout-type Theorem for Pfaffian Functions: A Self-Contained Proof, and Applications}, 
      author={Martin Lotz and Abhiram Natarajan},
      year={2026},
      eprint={2607.29267},
      archivePrefix={arXiv},
      primaryClass={math.AG},
      url={https://arxiv.org/abs/2607.29267}, 
}

@incollection{Marker1997,
  author={Marker, David},
  title={Khovanskii's Theorem},
  booktitle={Algebraic Model Theory},
  editor={Hart, Bradd T. and Lachlan, Alistair H. and Valeriote, Matthew A.},
  year={1997},
  publisher={Springer Netherlands},
  address={Dordrecht},
  pages={181--193},
  isbn={978-94-015-8923-9},
  doi={10.1007/978-94-015-8923-9_8},
  url={https://doi.org/10.1007/978-94-015-8923-9_8}
}

@misc{bickerton2026sharpnesskhovanskiisbezouttypebound,
      title={Parameterwise Sharpness of {K}hovanskii's {B}ezout-type Bound for {P}faffian Functions},
      author={Terence Bickerton and Joseph Harrison and Olivia Hornakova and Dominic Le-Mar and Abhiram Natarajan and Nadia Potter},
      year={2026},
      howpublished={arXiv:2606.24373},
      eprint={2606.24373},
      archivePrefix={arXiv},
      primaryClass={math.AG},
      url={https://arxiv.org/abs/2606.24373}, 
}

@book{cox2005using,
  title={Using algebraic geometry},
  author={Cox, David A. and Little, John and O'Shea, Donal},
  volume={185},
  year={2005},
  publisher={Springer Science \& Business Media},
  address={New York, NY},
  edition={2},
  series={Graduate Texts in Mathematics},
  isbn={978-0-387-20706-3}
}

@inproceedings{szegedy2013intriguing,
  author       = {Christian Szegedy and
                  Wojciech Zaremba and
                  Ilya Sutskever and
                  Joan Bruna and
                  Dumitru Erhan and
                  Ian J. Goodfellow and
                  Rob Fergus},
  editor       = {Yoshua Bengio and
                  Yann LeCun},
  title        = {Intriguing properties of neural networks},
  booktitle    = {2nd International Conference on Learning Representations, {ICLR} 2014,
                  Banff, AB, Canada, April 14-16, 2014, Conference Track Proceedings},
  year         = {2014},
  url          = {http://arxiv.org/abs/1312.6199},
  bibsource    = {dblp computer science bibliography, https://dblp.org}
}

@inproceedings{goodfellow2015explaining,
  author    = {Goodfellow, Ian J. and Shlens, Jonathon and Szegedy, Christian},
  title     = {Explaining and Harnessing Adversarial Examples},
  booktitle = {International Conference on Learning Representations},
  year      = {2015},
  eprint    = {1412.6572},
  archivePrefix = {arXiv}
}

@inproceedings{moosavi2016deepfool,
  author    = {Moosavi-Dezfooli, Seyed-Mohsen and Fawzi, Alhussein and Frossard, Pascal},
  title     = {{DeepFool}: A Simple and Accurate Method to Fool Deep Neural Networks},
  booktitle = {Proceedings of the IEEE Conference on Computer Vision and Pattern Recognition},
  pages     = {2574--2582},
  year      = {2016},
  doi       = {10.1109/CVPR.2016.282}
}

@inproceedings{fawzi2016robustness,
  author    = {Fawzi, Alhussein and Moosavi-Dezfooli, Seyed-Mohsen and Frossard, Pascal},
  title     = {Robustness of Classifiers: From Adversarial to Random Noise},
  booktitle = {Advances in Neural Information Processing Systems},
  volume    = {29},
  year      = {2016},
  eprint    = {1608.08967},
  archivePrefix = {arXiv}
}

@inproceedings{fawzi2018adversarial,
  author    = {Fawzi, Alhussein and Fawzi, Hamza and Fawzi, Omar},
  title     = {Adversarial Vulnerability for Any Classifier},
  booktitle = {Advances in Neural Information Processing Systems},
  volume    = {31},
  year      = {2018},
  eprint    = {1802.08686},
  archivePrefix = {arXiv}
}

@inproceedings{mahloujifar2019curse,
  author    = {Mahloujifar, Saeed and Diochnos, Dimitrios I. and Mahmoody, Mohammad},
  title     = {The Curse of Concentration in Robust Learning: Evasion and Poisoning Attacks from Concentration of Measure},
  booktitle = {Proceedings of the AAAI Conference on Artificial Intelligence},
  volume    = {33},
  pages     = {4536--4543},
  year      = {2019}
}

@book{burgisser2013condition,
  title={Condition: The geometry of numerical algorithms},
  author={B{\"u}rgisser, Peter and Cucker, Felipe},
  volume={349},
  year={2013},
  publisher={Springer Science \& Business Media}
}

@inproceedings{macintyre1993finiteness,
author = {Macintyre, Angus and Sontag, Eduardo Daniel},
title = {Finiteness results for sigmoidal ``neural'' networks},
year = {1993},
isbn = {0897915917},
publisher = {Association for Computing Machinery},
address = {New York, NY, USA},
url = {https://doi.org/10.1145/167088.167192},
doi = {10.1145/167088.167192},
booktitle = {Proceedings of the Twenty-Fifth Annual ACM Symposium on Theory of Computing},
pages = {325--334},
numpages = {10},
location = {San Diego, California, USA},
series = {STOC '93}
}

@article{karpinski1997polynomial,
  author    = {Marek Karpinski and Angus Macintyre},
  title     = {Polynomial Bounds for {VC} Dimension of Sigmoidal and General {P}faffian Neural Networks},
  journal   = {J. Comput. Syst. Sci.},
  volume    = {54},
  number    = {1},
  pages     = {169--176},
  year      = {1997},
  doi       = {10.1006/jcss.1997.1477},
  publisher = {Academic Press, Inc.},
  address   = {Orlando, FL, USA}
}

@article{bianchini2014complexity,
  title={On the complexity of neural network classifiers: A comparison between shallow and deep architectures},
  author={Bianchini, Monica and Scarselli, Franco},
  journal={IEEE transactions on neural networks and learning systems},
  volume={25},
  number={8},
  pages={1553--1565},
  year={2014},
  doi={10.1109/TNNLS.2013.2293637},
  publisher={IEEE}
}

@article{hendrycks2016gaussian,
  title={Gaussian Error Linear Units ({GELU}s)},
  author={Hendrycks, Dan and Gimpel, Kevin},
  journal={arXiv preprint arXiv:1606.08415},
  year={2016}
}

@inproceedings{glorot2011deep,
  title={Deep sparse rectifier neural networks},
  author={Glorot, Xavier and Bordes, Antoine and Bengio, Yoshua},
  booktitle={Proceedings of the fourteenth international conference on artificial intelligence and statistics},
  pages={315--323},
  year={2011},
  organization={JMLR Workshop and Conference Proceedings}
}

@article{lotz2020weak,
  title={Wilkinson's bus: Weak condition numbers, with an application to singular polynomial eigenproblems},
  author={Lotz, Martin and Noferini, Vanni},
  journal={Foundations of Computational Mathematics},
  volume={20},
  number={6},
  pages={1439--1473},
  year={2020},
  publisher={Springer}
}

@article{lotz2015volume,
  title={On the volume of tubular neighborhoods of real algebraic varieties},
  author={Lotz, Martin},
  journal={Proceedings of the American Mathematical Society},
  volume={143},
  number={5},
  pages={1875--1889},
  year={2015}
}

@article {basu2021hausdorff,
    AUTHOR = {Basu, Saugata and Lerario, Antonio},
     TITLE = {Hausdorff approximations and volume of tubes of singular
              algebraic sets},
   JOURNAL = {Math. Ann.},
  FJOURNAL = {Mathematische Annalen},
    VOLUME = {387},
      YEAR = {2023},
    NUMBER = {1-2},
     PAGES = {79--109},
      ISSN = {0025-5831,1432-1807},
       DOI = {10.1007/s00208-022-02458-w},
       URL = {https://doi.org/10.1007/s00208-022-02458-w},
}

@phdthesis{zell2003quantitative,
  title={Quantitative study of semi-{P}faffian sets},
  author={Zell, Thierry Paul},
  year={2003},
  school={Purdue University}
}

@book{milnor1997topology,
  title={Topology from the Differentiable Viewpoint},
  author={Milnor, John W. and Weaver, David W.},
  isbn={9780691048338},
  lccn={97030986},
  series={Princeton Landmarks in Mathematics and Physics},
  url={https://books.google.co.uk/books?id=BaQYYJp84cYC},
  year={1997},
  publisher={Princeton University Press}
}

@article{wilkie1999theorem,
  title={A theorem of the complement and some new o-minimal structures},
  author={Wilkie, Alex J},
  journal={Selecta Mathematica},
  volume={5},
  pages={397--421},
  year={1999},
  publisher={Springer}
}

@article{wilkie1996model,
  title={Model completeness results for expansions of the ordered field of real numbers by restricted {P}faffian functions and the exponential function},
  author={Wilkie, Alex},
  journal={Journal of the American Mathematical Society},
  volume={9},
  number={4},
  pages={1051--1094},
  year={1996}
}

@article {speissegger1999pfaffian,
    AUTHOR = {Speissegger, Patrick},
     TITLE = {The {P}faffian closure of an o-minimal structure},
   JOURNAL = {J. Reine Angew. Math.},
  FJOURNAL = {Journal f\"ur die Reine und Angewandte Mathematik. [Crelle's
              Journal]},
    VOLUME = {508},
      YEAR = {1999},
     PAGES = {189--211},
      ISSN = {0075-4102,1435-5345},
}

@incollection{gabrielov2004complexity,
  title={Complexity of computations with {P}faffian and {N}oetherian functions},
  author={Gabrielov, Andrei and Vorobjov, Nicolai},
  editor={Ilyashenko, Yulij and Rousseau, Christiane},
  booktitle={Normal Forms, Bifurcations and Finiteness Problems in Differential Equations},
  series={NATO Science Series II: Mathematics, Physics and Chemistry},
  volume={137},
  pages={211--250},
  year={2004},
  publisher={Kluwer Academic Publishers},
  address={Dordrecht}
}

@book{khovanskii1991fewnomials,
  title={Fewnomials},
  author={Khovanskii, Askold Georgievich},
  volume={88},
  year={1991},
  publisher={American Mathematical Soc.}
}

@article{weyl1939volume,
  title={On the volume of tubes},
  author={Weyl, Hermann},
  journal={American Journal of Mathematics},
  volume={61},
  number={2},
  pages={461--472},
  year={1939}
}

@article{steiner1840bestimmung,
  title={{\"U}ber parallele {F}l{\"a}chen},
  author={Steiner, Jakob},
  journal={Monatsberichte der Berliner Akademie der Wissenschaften},
  pages={114--118},
  year={1840}
}

@article{hotelling1939tubes,
  title={Tubes and spheres in $n$-spaces, and a class of statistical problems},
  author={Hotelling, Harold},
  journal={American Journal of Mathematics},
  volume={61},
  number={2},
  pages={440--460},
  year={1939}
}

@book{gray2004tubes,
  title={Tubes},
  author={Gray, Alfred},
  edition={2nd},
  year={2004},
  publisher={Birkh{\"a}user}
}

@article{bernstein1975,
  author  = {Bernstein, David N.},
  title   = {The number of roots of a system of equations},
  journal = {Functional Analysis and its Applications},
  volume  = {9},
  number  = {3},
  pages   = {183--185},
  year    = {1975}
}

@article{bihansottile2007,
  author  = {Bihan, Fr{\'e}d{\'e}ric and Sottile, Frank},
  title   = {New fewnomial upper bounds from {G}ale dual polynomial systems},
  journal = {Moscow Mathematical Journal},
  volume  = {7},
  number  = {3},
  pages   = {387--407},
  year    = {2007}
}

@inproceedings{montufar2014,
  author    = {Mont{\'u}far, Guido and Pascanu, Razvan and Cho, Kyunghyun and Bengio, Yoshua},
  title     = {On the Number of Linear Regions of Deep Neural Networks},
  booktitle = {Advances in Neural Information Processing Systems},
  volume    = {27},
  year      = {2014}
}

@inproceedings{hanin2019,
  author    = {Hanin, Boris and Rolnick, David},
  title     = {Complexity of Linear Regions in Deep Networks},
  booktitle = {Proceedings of the 36th International Conference on Machine Learning},
  year      = {2019}
}

@article{zhang2025covering,
  author  = {Zhang, Yifan and Kileel, Joe},
  title   = {Covering Number of Real Algebraic Varieties and Beyond: Improved Bounds and Applications},
  journal = {Foundations of Computational Mathematics},
  year    = {2025},
  doi     = {10.1007/s10208-025-09735-5},
  eprint  = {2311.05116},
  archivePrefix = {arXiv}
}

@inproceedings{raghu2017,
  author    = {Raghu, Maithra and Poole, Ben and Kleinberg, Jon and Ganguli, Surya and Sohl-Dickstein, Jascha},
  title     = {On the Expressive Power of Deep Neural Networks},
  booktitle = {Proceedings of the 34th International Conference on Machine Learning},
  pages     = {2847--2854},
  year      = {2017}
}

@article{dinverno2024vc,
  author  = {D'Inverno, Giuseppe Alessio and Bianchini, Monica and Scarselli, Franco},
  title   = {{VC} dimension of Graph Neural Networks with {P}faffian activation functions},
  journal = {Neural Networks},
  volume  = {182},
  pages   = {106924},
  year    = {2025},
  doi     = {10.1016/j.neunet.2024.106924},
  eprint  = {2401.12362},
  archivePrefix = {arXiv}
}

@inproceedings{serra2018,
  author    = {Serra, Thiago and Tjandraatmadja, Christian and Ramalingam, Srikumar},
  title     = {Bounding and Counting Linear Regions of Deep Neural Networks},
  booktitle = {Proceedings of the 35th International Conference on Machine Learning},
  year      = {2018},
  eprint    = {1711.02114},
  archivePrefix = {arXiv}
}

@inproceedings{zhang2018tropical,
  author    = {Zhang, Liwen and Naitzat, Gregory and Lim, Lek-Heng},
  title     = {Tropical Geometry of Deep Neural Networks},
  booktitle = {Proceedings of the 35th International Conference on Machine Learning},
  year      = {2018},
  eprint    = {1805.07091},
  archivePrefix = {arXiv}
}

@inproceedings{marchetti2025neuroalgebraic,
  author    = {Marchetti, Giovanni Luca and Shahverdi, Vahid and Mereta, Stefano and Trager, Matthew and Kohn, Kathl{\'e}n},
  title     = {Algebra Unveils Deep Learning: An Invitation to Neuroalgebraic Geometry},
  booktitle = {Proceedings of the 42nd International Conference on Machine Learning},
  series    = {PMLR},
  volume    = {267},
  year      = {2025},
  eprint    = {2501.18915},
  archivePrefix = {arXiv}
}

@article{cybenko1989approximation,
  author    = {Cybenko, George},
  title     = {Approximation by superpositions of a sigmoidal function},
  journal   = {Mathematics of Control, Signals and Systems},
  volume    = {2},
  number    = {4},
  pages     = {303--314},
  year      = {1989},
  publisher = {Springer}
}

@article{hornik1989multilayer,
  author    = {Hornik, Kurt and Stinchcombe, Maxwell and White, Halbert},
  title     = {Multilayer feedforward networks are universal approximators},
  journal   = {Neural Networks},
  volume    = {2},
  number    = {5},
  pages     = {359--366},
  year      = {1989},
  publisher = {Elsevier}
}

@article{barron1993universal,
  author    = {Barron, Andrew R.},
  title     = {Universal approximation bounds for superpositions of a sigmoidal function},
  journal   = {IEEE Transactions on Information Theory},
  volume    = {39},
  number    = {3},
  pages     = {930--945},
  year      = {1993},
  publisher = {IEEE}
}

@article{pinkus1999approximation,
  author    = {Pinkus, Allan},
  title     = {Approximation theory of the {MLP} model in neural networks},
  journal   = {Acta Numerica},
  volume    = {8},
  pages     = {143--195},
  year      = {1999},
  publisher = {Cambridge University Press}
}

\end{document}